\documentclass[11pt]{amsart}

\usepackage{enumerate}

\usepackage{amsfonts} 
\usepackage{amsmath}
\usepackage{amssymb}
\usepackage{mathtools}
\usepackage{enumitem}

\usepackage{mathptmx}       % selects Times Roman as basic font
\usepackage{helvet}         % selects Helvetica as sans-serif font
\usepackage{courier}        % selects Courier as typewriter font
\usepackage{type1cm}        % activate if the above 3 fonts are
\usepackage{graphicx}        % standard LaTeX graphics tool
\usepackage{multicol}        % used for the two-column index
\usepackage[bottom]{footmisc}% places footnotes at page bottom

\theoremstyle{plain}
\newtheorem{theorem}{Theorem}
\newtheorem{proposition}{Proposition}
\newtheorem{corollary}{Corollary}
\newtheorem{lemma}{Lemma}

\theoremstyle{definition}
\newtheorem{definition}{Definition}

\theoremstyle{remark}
\newtheorem{remark}{Remark}

\input xy
\usepackage[all]{xy}

\newcommand{\Proj} {\mathbb{P}}

\newcommand{\z}  {\mathbf{Z}}
\newcommand{\ci} {\mathbf{C}}

\newcommand{\qu} {\mathbf{Q}}

\newcommand{\letr}[1] {\mathcal{1}}

\newcommand{\gq} {G_{\mathbf{Q}}}
\newcommand{\f } {\mathbf{F}}

\DeclareMathOperator{\aut}{Aut}
\DeclareMathOperator{\gl}{GL}
\DeclareMathOperator{\Gal}{Gal}
\DeclareMathOperator{\pgl}{PGL}
\DeclareMathOperator{\End}{End}

\DeclareMathOperator{\fr}{Frob}

\def\p{\mathfrak{p}}
\def\n{\mathfrak{n}}

\def\b{\mathfrak{b}}
\begin{document}
\title[Integral Frobenius for Abelian Varieties]
{Integral Frobenius for Abelian Varieties with Real Multiplication}
%first author

\author{{Tommaso Giorgio} {Centeleghe}}
\address{Universit\"at Heidelberg, Im Neuenheimer Feld 368
69120 Heidelberg, Germany}
\email{tommaso.centeleghe@gmail.com }
%\urladdr{http://www.iwr.uni-heidelberg.de/groups/arith-geom/centeleghe/}

\author{{Christian} {Theisen}}
\address{
Universit\"at Heidelberg, Im Neuenheimer Feld 368\\
69120 Heidelberg, Germany}
\email{c.theisen@stud.uni-heidelberg.de}
\urladdr{}

\maketitle

\begin{abstract}
In this paper we introduce the concept of {\it integral Frobenius} to formulate an integral analogue of the classical compatibility
condition linking the collection of rational Tate modules $V_\lambda(A)$ arising from abelian varieties over number fields with
real multiplication. Our main result gives a recipe for constructing an integral Frobenius when the real multiplication field has
class number one. By exploiting algorithms already existing in the literature, we investigate this construction for three modular
abelian surfaces over $\qu$.
\end{abstract}

\section{Introduction}\label{sec:intro}
Let $K$ be a number field, and $A$ an abelian variety over $K$ with real multiplication. By this we shall mean throughout that it is given a totally real number field $E$
of degree $[E:\qu]=\dim (A)$ together with an embedding
\begin{equation}\label{eq:RMdef}
\iota: O_E\hookrightarrow\End_K(A)
\end{equation}
of its ring of integers $O_E$ in the ring of $K$-endomorphisms of $A$.
%\footnote{In the literature one can find a more general definition of real multiplication, in which the ring $O_E$
%is replaced by an order of $E$.}
To simplify our notation we will omit the reference to $\iota$, and regard $O_E$ as a given subring of the endomorphism ring of $A$.

Let $\lambda$ be any finite prime of $E$, denote by $E_\lambda$ the completion of $E$ at $\lambda$,
and by $O_\lambda\subset E_\lambda$ the corresponding valuation ring. Consider
the $\lambda$-adic Tate module
\[
T_\lambda(A) = \varprojlim_{n} A[\lambda^n],
\]
and its rational version
\[
V_\lambda(A) = T_\lambda(A)\otimes\qu.
\]
From the embedding $O_E\subseteq\End_K(A)$ one deduces a structure of $O_\lambda$-module on $T_\lambda(A)$ and a structure of $E_\lambda$-vector space on
$V_\lambda(A)$. The module $T_\lambda(A)$ is free of rank two over $O_\lambda$, and $V_\lambda(A)$ is two-dimensional over $E_\lambda$ (see \cite{Ri}, Prop. 2.2.1).
These structures are compatible with the action of the absolute Galois group $G_K=\Gal(\bar K/K)$, where $\bar K$ is a fixed algebraic closure of $K$. Thus
we have two Galois representations:
\[
\rho_\lambda :G_K\longrightarrow\gl_{O_\lambda}(T_\lambda(A))\simeq\gl_2(O_\lambda),
\]
\[
\rho^0_\lambda :G_K\longrightarrow\gl_{E_\lambda}(V_\lambda(A))\simeq\gl_2(E_\lambda).
\]

%Notice that $\rho_\lambda$ defines an $O_\lambda$-stable lattice of $\rho^0_\lambda$. If the residual representation
%
Let now $\p$ be a finite prime of $K$ where $A$ has good reduction $A_\p$. Denote by $k_\p$ be the residue field of $\p$, by $q$ its cardinality and by
$p$ its characteristic. Let $\fr_\p\in G_K$ be an arithmetic Frobenius element at $\p$. As it is well known, the representations
\[
{\{\rho^0_\lambda\}}_{\lambda\nmid p}
\]
are all unramified at $\p$ and satisfy a compatibility condition that can be formulated as follows.

\smallskip

\emph{There exists a semi-simple conjugacy class $\Sigma_\p^0\subset \gl_2(E)$ such that for every $\lambda\nmid p$ the image of $\Sigma_\p^0$ in $\gl_2(E_\lambda)$ defines
the conjugacy class of $\rho^0_\lambda(\fr_\p)$.}

Moreover, the characteristic polynomial of $\Sigma_\p^0$ has coefficients in $O_E$, and can be computed from the Frobenius isogeny $\pi_\p$ of $A_\p$
(see \S~\ref{sec:reduction}).

\smallskip

We find it natural to investigate an integral analogue of the property above. To this purpose we raise the following questions.

\begin{enumerate}

\item\label{Q1} Is there a conjugacy class
\[
\Sigma_\p\subset\gl_2(O_E[1/p])
\]
such that for any $\lambda\nmid p$ the action of $\fr_\p$ on $T_\lambda(A)$ is described by the conjugacy class of $\gl_2(O_\lambda)$ containing the
image of $\Sigma_\p$?

\item\label{Q2} If such a $\Sigma_\p$ exists, how can we describe it?
\end{enumerate}

\begin{definition}\label{IFdef} A conjugacy class $\Sigma_\p$ satisfying the requirement of question \ref{Q1} will be called an \emph{integral Frobenius} of $A$ at $\p$.
\end{definition}

\smallskip

In this paper, in the case where $E$ has class number one, we construct an explicit matrix $\sigma_\p\in\gl_2(O_E[1/p])$ with entries in $O_E$
and show that its conjugacy class is an integral Frobenius of $A$ at $\p$. Our main result generalizes a theorem of Duke and T\`oth (see \cite{DT},
Theorem 2.1) who treated the case of elliptic curves using a different technique. We remark that the integral questions raised above can be considered for
more general compatible systems of Galois representations arising from geometry.

\smallskip

This article is a development of the Master Thesis of one of us (C. Theisen) at the University~of~Heidelberg. We wish to thank Gebhard B\"ockle for all
the useful discussions we had on this project. We thank Gaetan Bisson for an intense email correspondence where he helped us understanding better
his work. The first author thanks Jordi Gu\`ardia and Josep Gonz\`alez for the warm hospitality he received during his visit to Universitat Polit\`ecnica
de Catalunya in July 2015. This work was supported by the DFG Priority Program SPP 1489 and the Luxembourg FNR.

\section{The main result}\label{sec:mainresult}

In this section we keep the notation of the introduction and further assume that $E$ has class number one.

The injectivity of the reduction map (see \cite{CCO}, \S~1.4.4)
\[
r_\p:\End_K(A)\to\End_{k_\p}(A_\p)
\]
will be used thoughout to identify the ring $O_E\subseteq\End_K(A)$ with a subring of $\End_{k_\p}(A_\p)$. In particular, if $\lambda$ is a prime of $E$ not dividing $p$,
we can make sense of the Tate modules $T_\lambda(A_\p)$ and $V_\lambda(A_\p)$, defined as in the characteristic zero case.

The ring $\End_{k_\p}(A_\p)$ has two distinguished elements (not necessarily distinct)
given by the Frobenius isogeny $\pi_\p:A_\p\to A_\p$ relative to $k_\p$ and the corresponding Verschiebung $q/\pi_\p:A_\p\to A_\p$. The existence of the embedding
$O_E\subseteq\End_{k_\p}(A_\p)$ implies that $A_\p$ is $k_\p$-isogenous to a power of a $k_\p$-simple abelian variety over $k_\p$ (see Proposition~\ref{prop:isotypical}).
Hence the $\qu$-subalgebra
\[
\qu(\pi_\p)\subseteq\End_{k_\p}(A_\p)\otimes\qu
\]
generated by $\pi_\p$ is a number field, and $\pi_\p$ a Weil $q$-number of it.

The element $\pi_\p$ plays a central role in the problem formulated in the introduction, in
that, by means of the natural identification
\[
T_\lambda(A_\p ) = T_\lambda (A),
\]
the action induced by $\pi_\p$ on $T_\lambda(A_\p)$ corresponds to the Galois action of the arithmetic Frobenius $\fr_\p$ on $T_\lambda(A)$.
The semi-simplicity of $\pi_\p$ acting on $V_\lambda(A_\p)$ (see \cite{Ta}, p. 138) can then be used to deduce that of $\fr_\p$ acting on $V_\lambda(A)$.
The characteristic polynomial of these $E_\lambda$-linear actions, denoted by
\[
h_\p(x) = x^2 -a_\p x+ s_\p,
\]
is independent of $\lambda$ and has coefficients in $O_E$ (see Proposition~\ref{prop:charpoly}).

\bigskip

After having recalled these basic facts, we give a recipe to construct an integral Frobenius $\sigma_\p\in\gl_2(O_E[1/p])$. The construction is divided in two
cases, according to whether the discriminant of $h_\p(x)$ is zero or not.

\begin{itemize}
\item $a_\p^2-4s_\p=0$. This condition is equivalent to $\pi_\p\in O_E$ (see Proposition~\ref{prop:charpoly}), and in this case the problem is trivial: since $\pi_\p$ acts on $T_\lambda(A)$ via scalar multiplication,
the matrix
\begin{equation}\label{equation:recipe1}
\sigma_\p = \begin{pmatrix}
\pi_\p & 0\\
0        & \pi_\p\\
\end{pmatrix}
\end{equation}
gives an integral Frobenius of $A$ at $\p$. Since $E$ is totally real, the Weil number $\pi_\p$ is square root of $q$ and
\[
h_\p(x) = {(x-\pi_\p)}^2=x^2 -2\pi_\p x + q.
\]

\item $a_\p^2-4s_\p\neq 0$. This is the interesting case of the problem, and the definition of $\sigma_\p$ is more involved. For more details we refer to \S~\ref{sec:quadord}.
The $E$-algebra $L=E[\pi_\p]\subseteq \End_{k_\p}(A_\p)\otimes\qu$ is semi-simple and has dimension two over $E$. Inside $L$ there
is a chain
of $O_E$-orders given by
\[
O_E[\pi_\p]\subseteq S_\p\subseteq O_L,
\]
where $O_L$ denotes the integral closure of $O_E[\pi_\p]$ in $L$, and $S_\p$ is defined as
\[
S_\p = E[\pi_\p] \cap \End_{k_\p}(A_\p),
\]
the intersection being taken in $\End_{k_\p}(A_\p)\otimes\qu$. The $O_E$-discriminant of $O_E[\pi_\p]$ is the principal ideal $(a_\p^2-4s_\p)$, which
can be written as
\[
(a_\p^2-4s_\p) = \delta_{O_L}\cdot \b_{O_L}^2,
\]
where $\delta_{O_L}$ is the $O_E$-discriminant of $O_L$, and $\b_{O_L}$ is the $O_E$-conductor of $O_E[\pi_\p]$ in $O_L$.
Let $\b_\p\subseteq O_E$ be the divisor of $\b_{O_L}$ corresponding to the intermediate order $S_\p$ (see Proposition~\ref{prop:inje}), and choose a generator $b_\p\in\b_\p$.
Let $u_\p\in O_E$ be any element such that the ratio $(\pi_\p-u_\p)/b_\p$ belongs to $S_\p$ (see Proposition~\ref{prop:basis}).
The matrix $\sigma_\p$ is defined by the formula 
\begin{equation}\label{equation:recipe2}
\sigma_\p = \begin{pmatrix}
u_\p  & -\dfrac{u_\p^2-a_\p u_\p + s_\p}{b_\p}\\
b_\p  & a_\p-u_\p \\
\end{pmatrix},
\end{equation}
from which it is easily checked that its characteristic polynomial is $h_\p(x)$.
\end{itemize}

\smallskip

\noindent Our main result says:

\begin{theorem}\label{theorem:main} The matrix $\sigma_\p$ has coefficients in $O_E$ and defines an integral Frobenius $\Sigma_\p$ of $A$ at $\p$.
\end{theorem}

In more concrete terms, the theorem says that for any finite prime $\lambda$ of $E$ not dividing $p$, the Tate module $T_\lambda(A)$ admits an $O_\lambda$-basis
such that the action of $\fr_\p$ on $T_\lambda(A)$ in the coordinates of this basis is given by $\sigma_\p$. In particular we deduce:
\begin{corollary} For any ideal $\n\subseteq O_E$ relatively prime to $\p$, the matrix $\sigma_\p$ describes the action of $\fr_\p$ on the $\n$-torsion points
$A[\n]$ of $A$, in the coordinates of a suitable $O_E/\n$-basis.
\end{corollary}
The corollary, which is essentially a reformulation of the main result, emphasises a connection of our work to \cite{DT}.

In the non-trivial case $a_\p^2-4s_\p\neq 0$,
write
\begin{equation}\label{eq:sigmadecomp}
\sigma_\p = \begin{pmatrix}
u_\p  &0\\
0 & u_\p \\
\end{pmatrix}
+ b_\p\begin{pmatrix}
0  & -\dfrac{u_\p^2-a_\p u_\p + s_\p}{b^2_\p}\\
1  & - \dfrac{2u_\p - a_\p}{b_\p} \\
\end{pmatrix}.
\end{equation}
From the construction of $\sigma_\p$ it follows that the matrix $(\sigma_\p-u_\p)/b_\p$ appearing in the right hand side of \eqref{eq:sigmadecomp} has coefficients in $O_E$
(see Proposition~\ref{prop:basis}). Thus from Theorem~\ref{theorem:main} we deduce the following interesting property of the ideal $\b_\p$.
For any prime-to-$p$ ideal $\n\subseteq O_E$ we have:
\[
\begin{gathered} \text{ $\fr_\p$ acts on $A[\n]$ as scalar}\\
\text{multiplication by an element in $O_E/\n$}
\end{gathered} \iff \text{ $\n$ divides $\b_\p$.}
\]
This equivalence can be linked to prime splitting phenomena in Galois extensions of number fields. Extend the definition of $b_\p$
by setting it equal to zero when $a_\p^2-4s_\p=0$. Let $\n$ be a nonzero ideal of $O_E$, and consider the projective Galois representation
\[
\Proj(\bar\rho_\n):G_K\longrightarrow\Proj\aut_{O_E/\n}(A[\n]) \simeq \pgl_2(O_E/\n)
\]
obtained from the $\n$-torsion of $A$. Let $K(\Proj A[\n])/K$ be the Galois extension of $K$ corresponding to the kernel of $\Proj(\bar\rho_\n)$.

\begin{corollary}\label{cor:split} Let $\p$ a prime of $K$ where $A$ has good reduction, and let $\n\subseteq O_E$ an ideal relatively prime to $p$.
Then $\p$ splits completely in $K(\Proj A[\n])/K$ if and only if $\n$ divides $\b_\p$.
\end{corollary}

In the case $a_\p^2-4s_\p\neq 0$, the matrix $\sigma_\p$, whose definition might appear quite mysterious,
represents the multiplication action of $\pi_\p$ on the ring $S_\p$ in the coordinates of a suitable $O_E$-basis (see \S~\ref{sec:quadord}, Remark~\ref{matrixsigmap}).
The main result relies on the key observation that the $\ell$-adic Tate module $T_\ell(A)$ is free of rank one over $S_\p\otimes\z_\ell$.
In \S~\ref{sec:reduction} we prove some basic facts on reduction in positive characteristic of abelian varieties
with real multiplication. In \S~\ref{sec:quadord} we discuss orders in quadratic extensions of number fields useful to understand the details of the construction
of $\sigma_\p$. In \S~\ref{sec:proof} we give the proof of Theorem~\ref{theorem:main}. Lastly, in \S\S~\ref{sec:comp}
and \ref{sec:tab} we exploit algorithms already existing in the literature (see \cite{Bis} and \cite{GGG}) to make computational investigations
with three modular abelian surfaces over $\qu$.

\section{Reduction of abelian varieties with real multiplication}\label{sec:reduction}
%
%In this section we prove some basic facts on the reductions in positive characteristic of abelian varieties with real multiplication. The results are
%well known to experts, we included them for the benefit of the reader.

We keep the notation and assumptions of the first two sections, so that $A$ is an abelian variety over a number field $K$ with
real multiplication by $E$, and $\p$ is a place of $K$ where $A$ has good reduction $A_\p$. We denote by $k_\p$ the residue field of $K$ at $\p$,
by $p$ its characteristic, and by $q=p^a$ its cardinality, where $a=[k_\p:\f_p]$. As before, $\pi_\p:A_\p\to A_\p$ denotes the Frobenius isogeny relative
to $k_\p$. The reduction of the real multiplication on $A$ gives inclusion
\begin{equation}\label{realmultp}
O_E\subseteq\End_{k_\p}(A_\p).
\end{equation}
The existence of this subring has the following consequence.\footnote{The fact that the variety $A_\p$ arises as reduction from characteristic zero plays
no role in Propositions~\ref{prop:isotypical} and \ref{prop:charpoly}.}

\begin{proposition}\label{prop:isotypical} The abelian variety $A_\p$ is isotypical, i.e., it is $k_\p$-isogenous to $B^n$, where $B$ is a $k_\p$-simple abelian variety
and $n$ is an integer $>0$.
\end{proposition}

\begin{proof} Consider a $k_\p$-isogeny
\[
f:A_\p\longrightarrow\prod_{1\leq i\leq h} B_i^{n_i}
\]
from $A_\p$ into the product of powers of $k_\p$-simple, pairwise non-$k_\p$-isogenous abelian varieties $B_i$, with $n_i>0$. We clearly have
$[E:\qu]=\dim(A_\p)=\sum_in_i\dim(B_i)$. To prove the lemma we have to show that $h=1$.

The isogeny $f$ induces an identification
\begin{equation}\label{equation:decEndAp}
\End_{k_\p}(A)\otimes\qu\simeq\prod_{1\leq i\leq h}M_{n_i}(D_i), % \End_{k_\p}^{\rm op}(B_i))
\end{equation}
where $M_{n_i}(D_i)$ is the ring of $n_i$-by-$n_i$ matrices with coefficients in the division ring $D_i=\End_{k_\p}(B_i)\otimes\qu$.

Let $\pi_i\in D_i$ be the Frobenius isogeny of $B_i$ relative to $k_\p$. The subfield $\qu(\pi_i)\subseteq D_i$ is the center of $D_i$ (see \cite{Ta}, Theorem~2 (a)), and a standard formula
from Honda-Tate theory says that
\begin{equation}\label{equation:dimformula}
2\dim(B_i) = s_i[\qu(\pi_i):\qu],
\end{equation}
where $s_i$ is the index of $D_i$, i.e., the square root of the degree $[D_i:\qu(\pi_i)]$.

The inclusion $E\subseteq\End_{k_\p}(A)\otimes\qu$ projects into each
factors of the decomposition \eqref{equation:decEndAp} and gives, for any $i$, an embedding
\begin{equation}\label{equation:emb}
\mu_i:E\longrightarrow M_{n_i}(D_i).
\end{equation}
We first complete the proof of the proposition assuming that there exist an index $i_0$ such that $\pi_{i_0}$ is not a
real Weil $q$-number.

Under this assumption we see that the compositum $L_{i_0}=\mu_{i_0}(E)\qu(\pi_{i_0})$ inside $M_{n_{i_0}}(D_{i_0})$
is a semi-simple commutative subalgebra of $M_{n_{i_0}}(D_{i_0})$ containing the center $\qu(\pi_{i_0})$ and strictly containing
the field $\mu_{i_0}(E)\simeq E$. Since the degree over $\qu(\pi_{i_0})$ of any commutative semi-simple subalgebra $L\subseteq M_{n_{i_0}}(D_{i_0})$
is bounded by $n_{i_0}s_{i_0}$, we conclude from \eqref{equation:dimformula} that
\[
2[E:\qu]\leq [L_{i_0}:\qu]\leq n_{i_0}2\dim(B_{i_0}),
\]
which readily implies that $h=1$, given that $[E:\qu]=\sum_i n_i\dim(B_i)$.

We are left with proving the proposition in the case where all Frobenius isogenies $\pi_i$ define real Weil $q$-numbers.
If $a$ is odd the proposition holds simply because there is \emph{only one} real Weil $q$-number, up to conjugation, namely that given
by $\sqrt q$, a real quadratic algebraic integer.

If $a$ is even there are precisely two distinct conjugacy classes of real Weil $q$-numbers, given by the integers $q^{a/2}$ and $-q^{a/2}$, and
the isogeny $f$ above has the form
\[
f:A\longrightarrow B_1^{n_1}\times B_2^{n_2},
\]
for some $n_1, n_2\geq 0$, where the Frobenius isogenies $\pi_1$ and $\pi_2$ are given by multiplication by $q^{a/2}$ and $-q^{a/2}$, respectively. As it turns out, both
$B_1$ and $B_2$ are supersingular elliptic curves (which are not $k_\p$-isogenous to each other) with all geometric endomorphisms defined
over $k_\p$. Their endomorphisms algebras $D_1$ and $D_2$ are both isomorphic to the definite $\qu$-quaternion $D$ ramified at $p$,
and we have $[E:\qu]=n_1+n_2$.

Arguing by contradiction, assume that both $n_1$ and $n_2$ are $>0$, and consider as above the two embeddings
\[
\mu_i:E\rightarrow M_{n_i}(D),
\]
for $i=1, 2$. Since $2n_i$ is the degree over $\qu$ of any commutative semi-simple subalgebra $L_i\subseteq M_{n_i}(D)$, we easily see that
$n_1=n_2$ and that $\mu_i(E)$ is maximal commutative subfield of $M_{n_i}(D)$. It follows that $\mu_1(E)$ is a splitting field for $M_{n_1}(D)$ and
hence it is also a splitting field for $D$. This is a contradiction since $\mu_1(E)$ is totally real, whereas every splitting field of the definite quaternion $D$
cannot have a real place. This completes the proof of the proposition. \qed
\end{proof}

Proposition~\ref{prop:isotypical} is equivalent to the statement that $\qu(\pi_\p)$ is a field (and not just a product of fields). In this way we see that $\pi_\p$ defines a
Weil $q$-number of the number field $\qu(\pi_\p)$. The complex conjugate of $\pi_\p$, with respect to any embedding $\qu(\pi_\p)\subset\ci$, is the Verschiebung
isogeny $q/\pi_p$ of $A_\p$.

Consider now the commutative subalgebra $E[\pi_\p]\subseteq \End_{k_\p}(A_\p)\otimes\qu$, and let $g_\p(x)$ be the minimal polynomial of $\pi_\p$ over $E$.
Since $\qu(\pi_\p)$ is semi-simple, so is $E[\pi_\p]$ and $g_\p(x)$ has non-zero discriminant. Moreover, the degree $[E[\pi_\p]:E]$ is either $1$ or $2$,
according to whether $\pi_\p$ belongs to $O_E$ or not, respectively. This can be seen using \eqref{equation:dimformula} and reasoning as in the proof of
Proposition~\ref{prop:isotypical}. Set now 
\[
h_\p(x) = \left\{\begin{matrix}
g_\p^2(x),\;\text{ if $\pi_\p\in O_E$}\\
g_\p(x),\;\text{ if $\pi_\p\not\in O_E$}\\
\end{matrix}\right..
\]
Since $\pi_\p$ is an algebraic integer, the polynomials $h_\p(x)$ and $g_\p(x)$ have coefficients in $O_E$. Moreover notice that $\pi_\p\in O_E$ if and only if $\pi_\p\in E$.

\begin{proposition}\label{prop:charpoly} Let $\lambda$ be a prime of $E$ not dividing $p$. The polynomial $h_\p(x)$ is the characteristic polynomial of the $E_\lambda$-linear
action induced by $\pi_\p$ on $V_\lambda(A)$. Its discriminant is zero if and only if $\pi_\p\in O_E$.
\end{proposition}
\begin{proof} If $\pi_E\in O_E$, then $V_\lambda(\pi_\p)$ is given by scalar multiplication by $\pi_\p$ itself. We have $g_\p(x)=(x-\pi_\p)$ and $h_\p(x)=g_\p^2(x)$.
If $\pi_E\not\in O_E$, then $g_\p(x)$ has degree two and thus we must have $h_\p(x)=g_\p(x)$. The last statement of the proposition follows from the fact that $g_\p(x)$
has distinct roots.
\qed
\end{proof}

\begin{remark}\label{remark:onhp}
Let $a_\p$ and $s_\p$ respectively denote the trace and the determinant of the $E_\lambda$-linear action induced by $\pi_\p$ on $V_\lambda(A_\p)$, so that we have
\[
h_\p(x) = x^2-a_\p x + s_\p.
\]
The following can be said about the coefficients of $h_\p(x)$. If $\pi_\p\in O_E$, then $\pi_\p$ is a real Weil $q$-number, hence its square is equal to $q$. In this case we have $a_\p=2\pi_\p$ and $s_\p=q$.
If $\pi_\p\not\in O_E$ and $\pi_\p$ is not real, then $h_\p(x)$ is irreducible in $E[x]$, and we have $a_\p=\pi_\p+q/\pi_\p$ and $s_\p=q$. Finally,
if $\pi_\p\not\in O_E$ and $\pi_\p$ is real, then $a_\p=0$ and $s_\p=-q$. This last case can only occur if $a$ is odd, and $h_\p(x)$ is reducible if and only if $O_E$ contains a square root of $q$.
\end{remark}

%Proposition \ref{prop:charpoly} shows how the characteristic polynomial of the semi-simple conjugacy class $\Sigma_\p^0\subset\gl_2(E)$ (see \S~\ref{sec:intro}) can be
%computed from $\pi_\p$.

\smallskip

We conclude the section with the following observation.

\begin{proposition} Assume that there is a place $\p$ of $K$ of good reduction for $A$ such that $\End_{k_\p}(A_\p)$ is commutative. Then
$E$ is the unique subfield of $\End_K(A)\otimes\qu$ which is totally real and has degree $\dim(A)$.
%This is to say that
%$A$ has a unique real multiplication.
\end{proposition}
\begin{proof} Let $E'\subseteq\End_{K}(A)\otimes\qu$ be a totally real number field with $[E':\qu]=\dim(A)$. We shall show that the image
of $E'$ in $\End_{k_\p}(A_\p)\otimes\qu$ under the reduction map (also denoted by $E'$) is equal to the number field $\qu(\pi_\p+q/\pi_\p)$,
which depends only on the reduction of $A$ modulo $\p$, and not on the choice of $E'$ inside $\End_K(A)\otimes\qu$.

The assumption on the place $\p$ is equivalent to ask that $A_\p$ be $k_\p$-simple, and that its endomorphism ring
tensored with $\qu$ be given by $\qu(\pi_p)$, for some non-real Weil $q$-number $\pi_\p$. Formula \eqref{equation:dimformula} from
Honda-Tate theory applied to the $k_\p$-simple variety $A_\p$ implies that
\[
\dim(A_\p)=[\qu(\pi_\p):\qu]/2 = [\qu(t_\p):\qu],
\]
where $t_\p=\pi_\p+q/\pi_\p$.

Arguing once again as in the proof of Proposition~\ref{prop:isotypical}, one can show that $E'\subseteq\End_{k_\p}(A_\p)\otimes\qu$ must contain $t_\p$.
Since $E'$ and $\qu(t_\p)$ have the same degree over $\qu$ they coincide, and the proposition follows. \qed
\end{proof}

\section{Quadratic orders}\label{sec:quadord}

In this section we clarify some aspects of the recipe given in \S~\ref{sec:mainresult} for the construction of the integral Frobenius
$\sigma_\p$ in the non-trivial case where $\pi_\p\not\in O_E$. 

Denote by $L$ the subalgebra $E[\pi_\p]\subseteq\End_{k_\p}(A_\p)\otimes\qu$ generated by $E$ and $\pi_\p$. In our notation for
$L$, for simplicity, we dropped any reference to the prime $\p$. Hopefully, this will not lead to any confusion.

Thanks to the assumption $\pi_\p\not\in O_E$, the polynomial $h_\p(x)$ has distinct roots (see Proposition~\ref{prop:charpoly}),
%the polynomial
%\[
%h_\p(x)=x^2-a_\p x + s_\p
%\]
%described in Proposition~\ref{prop:charpoly} has distinct roots and coincide with the minimal polynomial of $\pi_\p$ over $E$. Hence
and there is an isomorphism of $E$-algebras
\[
L\simeq E[x]/(h_\p(x)).
\]
Thus $L$ is either a quadratic field extension of $E$ or it is isomorphic to $E^2$, respectively according to whether $h_\p(x)$ is irreducible or not in $E[x]$.

In what follows by an \emph{$O_E$-order $S$ of $L$}, or simply an order of $L$, we shall mean a subring $S\subset L$ containing $O_E$ and defining an $O_E$-lattice
of $L$. Any such order $S$ is locally free of rank two over the localizations ${(O_E)}_\lambda$ of $O_E$ at each nonzero prime ideal $\lambda$.
There is a notion of \emph{$O_E$-discriminant} $\delta_S$ of an order $S\subset L$ (see \cite{Se2}, III \S 2). Without entering in the details here, we recall that
$\delta_S$ is an ideal of $O_E$ which, locally at any nonzero prime $\lambda$, is computed as the determinant of the usual bilinear
pairing given by the $E$-linear trace map
\[
(x, y)\mapsto\textrm{Tr}(xy).
\]
The $O_E$-discriminant of $O_E[\pi_\p]$ is generated by the discriminant $a_\p^2-4s_\p$ of $h_\p(x)$.

If $O_{L}$ denotes the integral closure of $O_E$ in $L$, we have a chain of inclusions of orders
\begin{equation}\label{eq:chain}
O_E[\pi_\p]\subseteq S_\p \subseteq O_{L},
\end{equation}
where
\[
S_\p = L\cap\End_{k_\p}(A_\p)
\]
is the order appearing in \S~\ref{sec:mainresult} in the definition of $\sigma_\p$. We observe that $O_{L}$ is the ring of integers of $L$ when $h_\p(x)$ is irreducible, and
it is isomorphic to $O_E^2$ otherwise.

Let now $S\subset L$ an order containing $\pi_\p$, and let $\mathfrak{b}_S\subseteq O_E$ the nonzero ideal given by
the annihilator of the torsion module $S/O_E[\pi_\p]$.

\begin{proposition}\label{prop:conduct} For any $O_E$-order $S\subset L$ containing $\pi_\p$ we have
\[
O_E[\pi_\p] = O_E + \mathfrak{b}_S S\;\;\text{ and }\;\;\delta_{O_E[\pi_\p]} = \delta_S\cdot\b_S^2
\]
The ideal $\mathfrak{b}_S$ will be called the {\rm $O_E$-conductor} of $O_E[\pi_\p]$ in $S$.
\end{proposition}
\begin{proof} Both equalities of the proposition can be proved after localization at each nonzero prime ideal $\lambda\subset O_E$, where the statements
becomes easy to verify since the localizations ${(O_E[\pi_\p])}_\lambda$ and $S_\lambda$ are free of rank two over the discrete
valuation ring ${(O_E)}_\lambda$. \qed
\end{proof}

The $O_E$-conductor of the order $S_\p$ entered in the recipe of the integral Frobenius from \S~\ref{sec:mainresult}, where it was denoted by $\b_\p$.
If $S, S'\subset L$ are orders containing $\pi_\p$ then from Proposition~\ref{prop:conduct} we deduce that $S\subset S'$ \emph{if and only if} $\b_S | \b_{S'}$.
In particular, we have that $\b_S | \b_{O_{L}}$ for any $S$. The next proposition shows that the invariant $\b_S$ suffices to
determine the order $S$.

\begin{proposition}\label{prop:inje} The map $\psi$ sending an $O_E$-order $S\subset L$ containing $\pi_\p$ to the conductor $\mathfrak{b}_S$
gives a bijection
\[
\psi:\left\{\begin{gathered}
	\text{$O_E$-orders $S\subset L$}\\
	\text{containing $\pi_\p$}
\end{gathered}\right\}
\stackrel{\sim}{\longrightarrow}
\left\{\begin{gathered}
\text{ideals $\mathfrak{b}\subseteq O_E$}\\
\text{dividing $\mathfrak{b}_{O_{L}}$}.
\end{gathered}\right\}
\]
\end{proposition}
\begin{proof} Let $S\subset L$ be an $O_E$-order containing $\pi_\p$.
%Since there is an inclusion $S/O_E[\pi_\p]\subseteq O_{L}/O_E[\pi_\p]$
%of torsion $O_E$-modules, the ideal $\mathfrak{b}_S$ divides $\mathfrak{b}_{O_{L}}$ and hence the map in the proposition is well defined.
Consider the short exact sequence of $O_E$-modules
\[
0\longrightarrow O_E[\pi_\p]\longrightarrow L\stackrel{r}{\longrightarrow}E/O_E\oplus (E/O_E)\cdot\bar\pi_\p\longrightarrow 0,
\]
where $\bar\pi_\p$ denotes the image of $\pi_\p$ in $L/E$. The quotient $S/O_E[\pi_\p]$ is a submodule of the right term of the sequence
which intersects the first summand trivially. Since there are isomorphisms of $O_E$-modules
\[
(E/O_E)\cdot\bar\pi_\p\simeq E/O_E\simeq\varinjlim_{0\neq \mathfrak{n}}O_E/\mathfrak{n},
\]
where the direct limit is taken over all nonzero ideals of $O_E$, we see that for any nonzero ideal $\mathfrak{b}\subset O_E$ there is a unique
submodule $M_\mathfrak{b}\subset (E/O_E)\cdot\bar\pi_\p$ whose annihilator is $\mathfrak{b}$. We conclude that
\[
S = r^{-1}(0\oplus M_{\mathfrak{b}_S}),
\]
and hence $S$ is uniquely determined by $\mathfrak{b}_S$ and $\psi$ is injective.

If $\mathfrak{b}\subseteq O_E$ is an ideal dividing $\mathfrak{b}_{O_{L}}$, then %it is not hard to see
\[
O_E+\dfrac{\mathfrak{b}_{O_{L}}}{\mathfrak{b}}O_{L}
\]
is an $O_E$-order of $L$ in which $O_E[\pi_\p]$ sits with conductor $\mathfrak{b}$. This shows that
$\psi$ is surjective and completes the proof of the proposition. \qed
\end{proof}

%The next lemma characterise the local evaluations of the $O_E$-conductor $\b_{O_{L}}$ of $O_E[\pi_\p]$ inside
%its integral closure $O_{L}$.
%

We assume for the rest of the section that $E$ has class number one.
%and prove a proposition useful to compute the
%$O_E$-conductor $\b_{O_{L}}$ of $O_E[\pi_\p]$ in $O_{L}$.
This assumptions, besides the principality of any ideal of $O_E$, ensures that any $O_E$-order $S\subset L$ is free of rank two as an
$O_E$-module (see \cite{Na}, Theorem 1.32).

\begin{proposition}\label{prop:basis} Assume that $E$ has class number one, and let $\b\subseteq O_E$ be a nonzero ideal. Then $\b$ divides $\b_{O_{L}}$ if and only
if there exists $u\in O_E$ such that the following conditions are satisfied:
\begin{enumerate}
\item\label{cond1} $h_\p'(u) = 2u-a_\p\in \b;$
\item\label{cond2} $h_\p(u) = u^2-a_\p u + s_\p \in \b^2.$
\end{enumerate}
Under these conditions, the reduction of $u$ modulo $\b$ is uniquely determined, and if $b$ is a generator of $\b$ the pair
\begin{equation}\label{basis}
(1, (\pi_\p-u)/b)
\end{equation}
is an $O_E$-basis of the order $S\subset L$ with $\b_S=\b$.
%Furthermore, if $\b$ is relatively prime to $2$, then the above conditions are both satisfied if and only if:
%\begin{enumerate}
%\item[1$'.$] $\b^2 | a_\p^2-4s_\p$.
%\end{enumerate}
\end{proposition}
\begin{proof} Reasoning as in the proof of Proposition~\ref{prop:inje}, we see that the ideal $\b$ divides $\b_{O_{L}}$ if and only if there exists $u\in O_E$ such
that the ratio $(\pi_\p-u)/b$ belongs to $O_{L}$, where $b$ is a generator of $\b$. This is to say that $\b$ divides $\b_{O_{L}}$ if and only if the minimal, monic
polynomial of $(\pi_\p-u)/b$ over $E$ has coefficients in $O_E$. Since this polynomial is given by
\[
x^2+\frac{2u-a_\p}{b}x +\frac{u^2-a_\p u + s_\p}{b^2},
\]
the first part of the proposition follows. This also shows that the pair \eqref{basis} is a basis of the order corresponding to $\b_S$ under the bijection $\psi$
from Proposition~\ref{prop:inje}. From this it is easy to see that $u$ is uniquely determined modulo $\b$. The proposition follows. \qed
\end{proof}

%\begin{remark} If $E$ is not assumed to have class number one, the first part of Proposition~\ref{prop:basis} remains valid in the special case where $\b$ is a power of a nonzero
%prime ideal $\lambda\subseteq O_E$. This can be proved using a local version of the proof we just gave.
%\end{remark}

\begin{remark}\label{matrixsigmap}
The matrix $\sigma_\p$ constructed in \S~\ref{sec:mainresult} represents the multiplication action of $\pi_\p$ on $S_\p$ on the coordinates induced by
an $O_E$-basis of the form $(1, (\pi_\p-u_\p)/b_\p)$, where $b_\p$ is a generator of $\b_\p$ and $u_\p$ is an element of $O_E$ chosen to satisfy the
two congruences of Proposition~\ref{prop:basis}.
\end{remark}

We point out the following corollary of Proposition~\ref{prop:basis}.

\begin{corollary}\label{bodd} Let $\b\subseteq O_E$ an ideal which is relatively prime to $(2)$. Then $\b$ divides
$\b_{O_L}$ if and only if $\b^2$ divides the discriminant $a_\p^2-4s_\p$ .
\end{corollary}
\begin{proof}
The ``only if'' part is clear from the relationship between discriminant and conductor. To see the if part, assume that $\b^2$ divides $(a^2_\p-4s_\p)$ and
let $u\in O_E$ be an element such that the first condition of the proposition is satisfied, i.e., 
\[
2u \equiv a_\p\text{ mod }\b.
\]
Such a $u$ exists since $\b$ and $(2)$ are relatively prime. Then
\[
4h_\p(u) = {(2u-a_\p)}^2 - (a_\p^2 - 4s_\p)
\]
is divisible by $\b^2$, since so are both summand. Since $\b$ and $(2)$ are relatively prime, we conclude that
$\b^2$ divides $h_\p(u)$ and the second condition of the proposition is also satisfied. Thus $\b$ divides $\b_{O_L}$.
\qed
\end{proof}

\bigskip 
We conclude the section with an observation that will be useful in our computations. Choose a generator $b_{O_L}$ of $\b_{O_{L}}$ and an element $u_\p$ such that the pair
\[
(1, \frac{\pi_\p-u_\p}{b_{O_L}})
\]
is an $O_E$-basis of $O_{L}$, and set $\mathbf{e}_2 = (\pi_\p-u_\p)/b_{O_L}$. From Propositions~\ref{prop:inje} and \ref{prop:basis} we deduce the following
corollary.
\begin{corollary}\label{cor:basis} Let $S\subset L$ an $O_E$-order containing $\pi_\p$, let $\b_S$ the $O_E$-conductor of $O_E[\pi_\p]$ in $S$, and let $b_S$ a generator of $\b_S$.
The pair
\begin{equation}\label{eq:basis}
(1, \frac{b_{O_L}}{b_S}\cdot\mathbf{e}_2) = (1,\frac{\pi_\p-u_\p}{b_S})
\end{equation}
is an $O_E$-basis of $S$.
\end{corollary}
As $b_S$ varies through the divisors of $b_{O_L}$, formula \eqref{eq:basis} parametrizes all $O_E$-orders $S\subset L$ containing $\pi_\p$, by
exhibiting $O_E$-basis of them.

\section{Proof of the main result}\label{sec:proof}

We first prove an abstract lemma that will be the key to our proof of Theorem~\ref{theorem:main}.
Let $R$ be a ring isomorphic to a finite product $\prod R_i$ of discrete valuation rings $R_i$, with total ring of
fractions $M$. Consider the free module $R^2$ of rank two, and assume that we are given an $R$-linear map $F:R^2\to R^2$
such that the $R$-subring $R[F]\subset\End_{R}(R^2)$ generated by $F$ is free of rank two as an $R$-module.

The map $F$ is given by a collection of $R_i$-linear maps $F_i:R_i^2\to R_i^2$, and the above requirement is equivalent to ask
that $F_i$ is not given by multiplication by an element of $R_i$, for every index $i$.
The ring
\[
S=\End_{R[F]}(R^2)
\]
of $R$-linear endomorphisms of $R^2$ commuting with $F$ is an order of $R[F]\otimes_R M$ containing $R[F]$ and which acts on
$R^2$ in the obvious way.

\begin{lemma}\label{lemma:key} $R^2$ is a free $S$-module of rank one.
\end{lemma}
\begin{proof} The ring $S$ decomposes as the product $\prod S_i$, where $S_i=\End_{R_i[F_i]}(R_i^2)$. Therefore the general form
of the lemma follows from the special case where $R$ is a discrete valuation ring, which we treat next. Denote by $\mathfrak{m}$ the maximal ideal
of $R$, choose a uniformizer $\omega$, and let $k$ be the residue field.

The $R$-order $S$ of the $M$-algebra $R[F]\otimes_R M$ is free of rank two over $R$, and therefore
\[
S=R\oplus R\cdot F_0,
\]
for some $F_0\in S$ which does not belong to the subring $R\subset S$. We claim that the morphism
\[
F_0\text{ mod }\mathfrak{m}:R^2/\mathfrak{m}R^2\longrightarrow R^2/\mathfrak{m}R^2
\]
is not given by multiplication by any element of $k$. For otherwise there exists $\lambda\in R$ such that
$F_0-\lambda$ sends $R^2$ to $\mathfrak{m}R^2$. This implies that $(F_0-\lambda)/\omega\in S$, which contradicts
the fact that $(1, F_0)$ is an $R$-basis of $S$.

The claim says precisely that there exists $r\in R^2\setminus\mathfrak{m}R^2$ such that
\[
F_0(r)\not\in R\cdot r + \mathfrak{m}R^2.
\]
From Nakayama's Lemma we deduce that $(r, F_0(r))$ is an $R$-basis of $R^2$, since the reductions of its elements generate
$R^2/\mathfrak{m}R^2$. From this it readily follows that the map
\[
S\ni s\mapsto s(r)\in R^2
\]
is an isomorphism of $S$-modules. This completes the proof of the lemma. \qed
\end{proof}

%We are now ready to show that the pair $(T(\p), F(\p))$ constructed in \S~\ref{sec:recipe} is an integral Frobenius for $A$
%at $\p$.

We now give the proof of Theorem~\ref{theorem:main}, the main result of the paper.

\begin{proof} %[of Theorem~\ref{theorem:main}]
%Let $\ell$ be a prime number $\neq p$, we have to show that there exists an
%$O_E\otimes\z_\ell$-linear isomorphism $T(\p)\otimes\z_\ell\simeq T_\ell(A_\p)$ for which the operator $F(\p)\otimes\z_\ell$ corresponds to $T_\ell(\pi_\p)$,
%the map induced by the Frobenius isogeny. The theorem will then follow since $T_\ell(\pi_\p)$ coincide with the action of $\fr_\p\in G_K$
%on $T_\ell(A)$, once the natural identification $T_\ell(A_\p)\simeq T_\ell(A)$ is made.
%
%If $\pi_\p$ belongs to $E$, then for any prime $\ell\neq p$ the arithmetic Frobenius $\rho_\ell(\fr_\p)$ acts on $T_\ell(A)$ as multiplication by
%$\pi_\p$. The theorem follows easily in this case and we continue the proof assuming $\pi_\p\not\in E$.
%
%Since $O_E$ is a Dedekind domain, the $O_E$-module underlying $S_\p$ is isomorphic to $O_E\oplus I$, for some nonzero ideal
%$I\subset O_E$ (see \cite{Na}, Theorem 1.32). From this we deduce that $T(\p)$ is locally free and hence projective.
%%Notice that $T(\p)$ is free if $O_F$ has class number one.
The result is trivial if $\pi_\p\in E$, therefore we continue assuming $\pi_\p\not\in E$.
Let $\ell$ be a prime different from $p$, the residual characteristic of $\p$.
By a well known result of Tate (see \cite{Ta}), there is a natural isomorphism
\[
r^0_\ell:\End^0_{k_\p}(A_\p)\otimes\qu_\ell\stackrel{\sim}{\longrightarrow}\End_{\qu_\ell[\pi_\p]}(V_\ell(A)).
\]
Since $\pi_\p\not\in E$, the subalgebra $L=E[\pi_\p]\subseteq \End^0_{k_\p}(A_\p)$ is a maximal commutative semi-simple subring, and
hence it coincides with its own commutator. This implies that the restriction
of $r^0_\ell$ to $L\otimes\qu_\ell$ induces an isomorphism
\begin{equation}\label{eqisomtate1}
s^0_\ell:L\otimes\qu_\ell\stackrel{\sim}{\longrightarrow}\End_{L\otimes\qu_\ell}(V_\ell(A_\p).
\end{equation}
Now, the integral version of $r^0_\ell$, which is given by
\begin{equation}\label{eqisomtate2}
r_\ell:\End_{k_\p}(A_\p)\otimes\z_\ell\stackrel{\sim}{\longrightarrow}\End_{\z_\ell[\pi_\p]}(T_\ell(A)),
\end{equation}
is also an isomorphism. From \eqref{eqisomtate1} and \eqref{eqisomtate2} we conclude that the map
\[
s_\ell:S_\p\otimes\z_\ell\longrightarrow\End_{S_\p\otimes\z_\ell}(T_\ell(A_\p))
\]
arising as the restriction of $r_\ell$ to $S_\p\otimes\z_\ell$ is also an isomorphism. Since $T_\ell(A_\p)$ is free of rank two over
$O_E\otimes\z_\ell$ and $\pi_\p\not\in O_E$, Lemma~\ref{lemma:key} gives that $T_\ell(A_\p)$ is a free $S_\p\otimes\z_\ell$-module
of rank one\footnote{More generally, this freeness holds if $S_\p\otimes\z_\ell$ is a Gorenstein ring (see \cite{ST}, Remark p. $502$).},
and hence
\begin{equation}
\text{$T_\lambda(A_\p)$ is a free $S_\p\otimes_{O_E} O_\lambda$-module of rank one.}
\end{equation}
Theorem~\ref{theorem:main} now follows from the fact that the matrix $\sigma_\p$ describes, by construction, the multiplication action of $\pi_\p$ on $S_\p$ in a suitable basis. \qed
\end{proof}

\section{Computations}\label{sec:comp}

Our aim in the remaining part of the paper is to explain how two algorithms already present in the literature (see \cite{GGG} and \cite{Bis}) can be used to
compute the integral Frobenius at several primes of good reduction for certain modular abelian \emph{surfaces} over $\qu$. We are grateful to the authors of these
algorithms for providing us with the nice opportunity to make experimental tests. All our auxiliary computations, like those in \cite{GGG} and \cite{Bis}, have been
performed using Magma (see \cite{Mg}).

\subsection{The main algorithms}
The first algorithm is the result of joint work of Gonz\'alez-Jimen\'ez, Gonz\'alez and Gu\`ardia (see \cite{GGG}). The input from which they start
is a cuspidal, normalized eigenform $f=\sum a_nq^n\in S_2(\Gamma_0(N))$ of weight $2$, trivial nebentype and conductor $N$ such that its Fourier
coefficient field $E_f$ is a (real) quadratic extension of $\qu$. The modular abelian surface $A_f$ attached to $f$ via the classical Shimura construction (see \cite{Sh}) is
a $\qu$-subvariety of the Jacobian $\textrm{Jac}X_0(N)$ of the modular curve $X_0(N)$, and has good reduction away from $N$. The Hecke action induces an inclusion
\begin{equation}\label{hecke}
O_f\subseteq\End_\qu(A_f),
\end{equation}
where $O_f=\z[{(a_p)}_{p\nmid N}]$ is the order of $E_f$ generated by the Fourier coefficients of $f$ indexed by primes not dividing $N$.

Assuming that the canonical polarization on $A_f$ coming from that of the Jacobian of the modular curve $X_0(N)$ is a power of a principal one, the three authors
compute a hyperelliptic, genus two equation
\[
y^2 = F(x),
\]
where $F(x)\in\z[x]$ has degree $5$ or $6$, whose desingularization defines a curve $C_f$ over $\qu$ such that there is an isomorphism
\begin{equation}\label{isom:jac}
\textrm{Jac}(C_f) \simeq A_f 
\end{equation}
of principally polarized abelian varieties over $\qu$. In table at the end of their paper they list the hyperelliptic equations that they obtained for the $78$
modular abelian surfaces of conductor $\leq 500$ whose canonical polarization satisfies the required condition. We remark that their output, and hence also ours,
is correct only up to numerical approximation. However, several tests in favour of its correctness are performed by the authors.

Notice that if \eqref{hecke} extends to the whole ring of integers $O_{E_f}\subset E_f$, then $A_f$ is an abelian surface with real multiplication by $E_f$, according
to the definition we gave in \S~\ref{sec:intro}. Furthermore, if $E_f$ has class number one then it makes sense to try and compute the integral
Frobenius of $A_f$ at primes $p\nmid N$.

\bigskip

The second algorithm on which our computations depend is due to Bisson (see \cite{Bis}). The input is a smooth genus two curve $C$ over a finite
field $\f$ with $q$ elements such that its Jacobian $\textrm{Jac}(C)$ is an absolutely simple, ordinary abelian surface over $\f$. The curve is assumed
to be represented by a hyperelliptic equation $y^2=\bar F (x)$, for a suitable polynomial $\bar F(x)\in\f[x]$ of degree $5$ or $6$.
Under these assumptions, the algorithm returns the endomorphism ring of the principally polarized abelian surface given by $\textrm{Jac}(C)$, which
is an order of the quartic number field $\qu(\pi)$ generated by the Frobenius isogeny $\pi$ of $\textrm{Jac}(C)$ relative to $\f$.

\subsection{Synthesis of the algorithms} The strategy we suggest for computing the integral Frobenius at primes of good reduction
for a modular abelian surface $A_f$ over $\qu$ consists of the following steps.

\begin{enumerate}
\item\label{step1} Start from a cuspidal, normalized eigenform $f\in S_2(\Gamma_0(N))$ whose coefficient field $E_f$ is quadratic. The first goal is to use \cite{GGG} to
find a hyperelliptic equation of a genus two curve $C_f$ over $\qu$ such that the isomorphism \eqref{isom:jac} holds.
\end{enumerate}

\noindent There are $465$ modular surfaces of conductor $\leq 500$. In the $75$ cases where the canonical polarization is a power of a principal one, \cite{GGG}
provides the hyperelliptic equations of the corresponding curves $C_f$. In the remaining cases, one can still try to use the same algorithm to solve \eqref{isom:jac} in $C_f$
by constructing a principal polarization on $A_f$. In \cite{GGR} this problem in carefully analized and sufficient conditions for the existence of $C_f$ are given.

\smallskip

\noindent We continue assuming that step \ref{step1} was successful, and perform now two checks. 

\begin{enumerate}[resume]
\item Check that the inclusion \eqref{hecke} extends to the ring of integers $O_{E_f}$.
\end{enumerate}

\noindent This maximality condition is often satisfied in practice.
With the help of Magma we verified that for $428$ modular surfaces of conductor $\leq 500$ the order
$O_f$ is already the maximal order of $E_f$. Moreover, using \cite{Bis}, we verified that for only two of the surfaces $A_f$ considered in \cite{GGG} the ring
$\End_{\qu}(A_f)$ fails to be the maximal order. These surfaces are those with conductor $224$, where $\End_{\qu}(A_f)$ sits in
$O_{E_f}$ with index two.

\begin{enumerate}[resume]
\item Check that the class number of $E_f$ is one.
\end{enumerate}

\noindent This condition is required by our method for constructing integral Frobenia. Among surfaces of conductor $\leq 500$ the condition fails only once in
conductor $276$.

\bigskip

\noindent Assuming that the three steps above are successfully completed, we enter now the second part of the strategy. Let $p$ be a prime $\nmid N$, denote by
$A_{f,p}$ the reduction of $A_f$ at $p$, and by $\pi_p$ the Frobenius isogeny of $A_{f,p}$ relative to its base field $\f_p$.
%Following the recipe of Theorem~\ref{theorem:main}
%we now attempt the computation of the matrix $\sigma_p$ describing an integral Frobenius of $A_f$ at $p$.
By the Eichler-Shimura relation, we have
\[
\pi_p+p/\pi_p = a_p \in O_{E_f},
\]
where $a_p$ is the $p$-th Hecke eigenvalue of $f$, and hence the characteristic polynomial of $\sigma_p$ is given by
\begin{equation}\label{equation:polyh}
h_p(x) = x^2 - a_p x + p.
\end{equation}
If $a_p^2-4p=0$, then
\[
\pi_p\in O_{E_f}\subseteq\End_{\f_p}(A_{f,p}),
\]
and the integral Frobenius $\sigma_p$ is the scalar matrix given by multiplication by $\pi_p$. We remark that in the computation we performed we
never run in such an example.

\noindent We therefore continue assuming $a_p^2-4p\neq 0$, which also implies that $\pi_p$ is not a real Weil $p$-number, for otherwise we
would have $h_p(x)=x^2-p$ (see \S~\ref{sec:reduction}, Remark \ref{remark:onhp}), a contradiction to \eqref{equation:polyh}.

\begin{enumerate}[resume]
\item\label{step4} Consider the quadratic $E_f$-algebra
\[
L=E_f[\pi_p]\subseteq\End_{\f_p}(A_{f,p})\otimes\qu,
\]
and compute the ideal $\b_{O_L}$ given by the $O_{E_f}$-conductor of $O_{E_f}[\pi_p]$ in its integral closure $O_L\subset L$. Compute further
a generator $b_{O_L}$ of $\b_{O_L}$ and an element $u_p\in O_{E_f}$ such that the element
\[
\mathbf{e}_2 = \frac{\pi_p-u_p}{b_{O_L}}
\]
completes $1\in O_{E_f}$ to an $O_{E_f}$-basis of $O_L$.
\end{enumerate}

\noindent Using Propositions~\ref{prop:conduct}, \ref{prop:basis} and Corollary~\ref{bodd}, the required computation can be carried out
using basic Magma functions on the arithmetic of real quadratic fields. Notice that the $O_{E_f}$-basis $(1, \mathbf{e}_2)$ of $O_L$
satisfies the useful property of Corollary~\ref{cor:basis}.

%If $b_S$ is a divisor of $b_{O_L}$, then the pair
%\[
%(1, \frac{b_{O_L}}{b_S}\cdot\mathbf{e}_2)
%\]
%is an $O_{E_f}$-basis of the order $S$ in which $O_{E_f}[\pi_p]$ sits with conductor $(b_S)$.

\smallskip

\noindent The crucial information that we need to compute for the recipe of the integral Frobenius is the ideal $\b_p$ given by the $O_{E_f}$-conductor of
$O_{E_f}[\pi_p]$ inside $S_p$, where $S_p$ is the order $L\cap\End_{\f_p}(A_{f,p})$. If the conductor $\b_{O_L}$ is the
trivial ideal $O_{E_f}$, then the chain \eqref{eq:chain} becomes
\[
O_{E_f}[\pi_p] = S_p = O_L,
\]
and hence the ideal $\b_p$ is trivial, and the integral Frobenius is simply given by
the companion matrix
\begin{equation}\label{eq:compmatrix}
\sigma_p=\begin{pmatrix}
0 & -p\\
1 & a_p\\
\end{pmatrix}.
\end{equation}
We then continue assuming that the ideal $\b_{O_L}$ is a proper ideal of $O_E$. In this case there is more than one possibility for the order
$S_p$, and to decide which one occurs we want to use Bisson's algorithm to compute the ring $\End_{\f_p}(A_{f,p})$.
In order to do so we first have to make sure that the assumptions of his algorithm are satisfied. We discuss these in the next three steps. If one
of these assumptions fails, then our strategy will not lead to the computations of the integral Frobenius of $A_f$ at $p$.

\begin{enumerate}[resume]
\item In the case where $a_p^2-4p\neq 0$ and the ideal $\b_{O_L}$ is proper, check whether the affine model $\z[x,y]/(y^2-F(x))$ of $C_f$ has good reduction
at the prime $p$.
\end{enumerate}

\noindent It can happen that the model of $C_f$ coming from the algorithm in \cite{GGG} has singular reduction at a prime $p\nmid N$. In our computations
this never occurred in a case where $a_p^2-4p\neq 0$ and $\b_{O_L}\subsetneq O_{E_f}$. 

\begin{enumerate}[resume]
\item In the case where $a_p^2-4p\neq 0$ and the ideal $\b_{O_L}$ is proper, check whether the abelian surface $A_{f,p}$ is ordinary.
\end{enumerate}

\noindent Recall that a Weil $p$-number $\pi$ is ordinary if and only if the algebraic integer $\pi+p/\pi$ is a $p$-adic unit. In our case this amounts
to check if $a_p$ is relatively prime to $p$ in $O_{E_f}$, which can easily be done in Magma.

\begin{enumerate}[resume]
\item In the case where $a_p^2-4p\neq 0$ and the ideal $\b_{O_L}$ is proper, check whether the abelian surface $A_{f,p}$ is absolutely irreducible.
\end{enumerate}

\noindent The abelian surface $A_{f,p}$ is either $\f_p$-isogenous to the square of an elliptic curve or it is $\f_p$-simple (see Proposition~\ref{prop:isotypical}).
Since our base field is $\f_p$ and the Weil number $\pi_p$ is not real, we know from Honda-Tate theory (see \cite{T})
that
\[
\text{$A_{f,p}$ is $\f_p$-simple $\iff$ $\pi_p+p/\pi_p=a_p\not\in\z$ $\iff$ $[\qu(\pi):\qu]=4$}.
\]
In order to proceed we then must require $a_p\not\in\z$ and still have to check whether $A_{f,p}$ is absolutely simple or not. This amount to verify the equality of number fields
\begin{equation}\label{equation:equality}
\qu(\pi_p^N) = \qu(\pi)
\end{equation}
for any integer $N\geq 2$. Since $\qu(\pi)$ is a CM quartic extension of $\qu$ it suffices to check \eqref{equation:equality}
for all integers $N\geq 2$ such that $\varphi(N)\leq 4$, where $\varphi$ denotes the Euler $\varphi$-function. These values are $2, 3, 4, 5, 6, 10$ and $12$,
and \eqref{equation:equality} can be verified using Magma.

\begin{enumerate}[resume]
\item In the case where $a_p^2-4p\neq 0$ and the ideal $\b_{O_L}$ is proper, assuming that the affine model of $C_f$ has good reduction at $p$ and the
abelian surface $A_{f,p}$ is absolutely irreducible and ordinary, use \cite{Bis} to compute $\End_{\f_p}(A_{f,p})$. Then extract from it the information of the ideal $\b_p$.
\end{enumerate}

\noindent Since $A_{f,p}$ is $\f_p$-simple and $\pi_p$ is not real, we have that the CM quartic field $\qu(\pi_p)$ coincides with the algebra $\End_{\f_p}(A_{f,p})\otimes\qu$.
So that the ring of $\f_p$-endomorphisms of $A_{f,p}$ is an order of $\qu(\pi_p)$. Notice that we also have
\[
L = \qu(\pi)\text{ and }S_p=\End_{\f_p}(A_{f,p}).
\]
%using the notation of Step~\ref{step4}.
The output of Bisson's algorithm is a $\z$-basis of the $\z$-order $\End_{\f_p}(A_{f,p})\subset\qu(\pi_p)$, expressed in terms of the basis
$(1, \pi, \pi^2, \pi^3)$ of $\qu(\pi_p)$. We are left with converting this output in an ``$O_{E_f}$-linear'' format, suitable for our purposes. To do this we use the equality
$\pi_p+p/\pi_p=a_p$ to embed the totally real field $E_f$ in the number field $\qu(\pi)$. In this way, using step \ref{step4} and Corollary~\ref{cor:basis},
we can control all $O_{E_f}$-orders $S$ containing $O_{E_f}[\pi_p]$ by exhibiting for each of them an $O_{E_f}$-basis of the form
\begin{equation}\label{equation:final}
(1, \frac{b_{O_L}}{b_S}\cdot\mathbf{e}_2)
\end{equation}
inside the number field $\qu(\pi)$ where Bisson's output lives. Letting $b_S$ vary through a set of generators of all divisors of $\b_{O_L}$, we
can then easily determine the unique element for which the $O_{E_f}$-span of the pair \eqref{equation:final} gives the lattice from Bisson's algorithm.
This element is the generator $b_p$ of the ideal $\b_p$ we were after.

%\begin{enumerate}
%
%
%
%
%
%and has class number one;
%
%
%\item[2)] Assume that the polarization on the surface $A_f$ is principal;
%\item[2)] Check that the injection \eqref{hecke} extends to $O_{E_f}$;
%\item[iii)] the class number of $E_f$ is one.
%\item[2)]
%\end{enumerate}

\section{Tables of results}\label{sec:tab}

We applied the strategy explained in the previous section to three modular abelian surfaces $A_f$ over $\qu$. The goal is to compute as many integral Frobenia as possible
at primes $p$ of good reduction in the range $2,\ldots, 1997$. The three Hecke cuspidal newforms $f$ we chose have conductor $N=23, 125$ and $133$. They all
lie in the first Galois orbit of the corresponding space $S_2(\Gamma_0(N))$, according to Magma enumeration. The hyperelliptic equations we used for the genus two curves
$C_f$ are those appearing in \cite{GGG}.

%
%Since the
%rational prime $2$ is inert in $\qu(\sqrt 5)$, we have that, for each $f$, the Galois representation coming from the $2$-torsion of $A_f$
%\[
%\bar\rho_{(2)}:\gq\longrightarrow\aut_{O_{E_f}}(A_f[2])\simeq\gl_2(\f_4)
%\]
%

The heading of the six columns of each table follows the notations of the paper. The first column consists of primes $p$. The second, third and fourth columns
($a_p$, $u_p$ and $b_p$) contain the elements of $O_{E_f}$ needed to construct the integral Frobenius $\sigma_p$, they are expressed with respect to the
$\qu$-basis $(1, a)$ of $E_f$ used by Magma to parametrized the coefficient field $E_f$. The fifth and sixth columns respectively give the prime factorizations of
the ideals $\b_p$ and $\b_{O_L}$ in $O_{E_f}$.\footnote{These ideals are always defined in our computations as we never found a prime $p$ for which $a_p^2-4p=0$.}
If $\ell$ is a rational primes which is does not split in $E_f$, then the corresponding prime of $O_{E_f}$ is denoted by $(\ell)$ or $\lambda_\ell$, according to whether
$\ell$ is inert or ramifies, respectively. If $\ell$ is split, then the corresponding primes are denoted by $\lambda_{\ell, 1}$ and $\lambda_{\ell, 2}$.

In every table we listed all primes $p\leq 1997$ where the given surface has good reduction and such that the order $O_{E_f}[\pi_p]$ is not the maximal order of $L=E[\pi_p]$.
When we were not able to apply Bisson's algorithm (or when the algorithm did not terminate), a dash (-) appears in place of the entries $u_p$ and $b_p$. In certain
cases we did obtain an output from Bisson's algorithm even though its basic assumptions on the input surface $A_{f,p}$ were not satisfied. These primes appear marked
in the tables: the symbol $(*)$ indicates that $A_{f,p}$ is not ordinary, and the symbol $(**)$ denotes that it is not absolutely simple, but just $\f_p$-simple.

Finally, in every example considered, the coefficient field $E_f$ is the real quadratic field $\qu(\sqrt 5)$ of discriminant $5$, and the order $O_f$ is the maximal order.
In the last two examples, the Galois representation on the $2$-torsion $A_f[2]$ defines two extensions of $\qu$ with Galois group isomorphic to $A_5$, the alternating group
in $5$ letters. The computation of the integral Frobenius, when successful, reveal the primes that are completely split in these extensions.

\subsection{First example}

Let $f\in S_2(\Gamma_0(23))$ be the unique normalized cusp form of weight $2$ and level $23$. The element $a\in E_f$ used by Magma to parametrize $E_f$
has minimal polynomial $x^2+x-1$. The first few coefficients of the Fourier expansion of $f$ are
\[
f= q + aq^2 - (2a+1)q^3 - (a+1)q^4 + 2aq^5+\ldots
\]
In Table~\ref{tab:1} we can experimentally observe a reducibility phenomenon predicted by a famous result of Mazur (see \cite{Maz}): since the prime
$11$ divides $N-1$, Mazur predicts the existence of a prime $\lambda$ of $E_f$ lying above $11$ such that
\[
\bar\rho_\lambda \simeq \textrm{1}\oplus\chi_{11},
\]
where $\bar\rho_\lambda$ is the residual Galois representation of $\rho_\lambda$, $\chi_{11}$ is the mod $11$ cyclotomic character, and $1$ is the trivial
character. The consequence of this result relevant for our computation is that for every prime $p\neq 23$ with $p\equiv 1\text{ mod 23}$ the ideal $\b_p$
appearing in the definition of the integral Frobenius is divisible by $\lambda$. Such ideal $\lambda$ is denote by $\lambda_{11,1}$ in the table.

\begin{table}[ht]
\caption{Integral Frobenius for $A_f$, where $f\in S_2(\Gamma_0(23))$}
\label{tab:1}       % Give a unique label
%
% Follow this input for your own table layout
%
\begin{tabular}{p{1.5cm}p{1.7cm}p{1.5cm}p{1.5cm}p{1.5cm}p{1.5cm}}
\hline\noalign{\smallskip}
$p$ & $a_p$ & $u_p$ & $b_p$ & $\textrm{Fac}(\b_p)$ & $\textrm{Fac}(\b_{O_L})$  \\
\hline\noalign{\smallskip}
$19$ & $-2$ & $-$ & $-$ & $-$ & $ (3)$ \\
$43$ & $0$ & $-$ & $-$ & $-$ & $ (2)$ \\
$53$ & $-2+4a$ & $0$ & $1$ & $(1)$ & $ (2)$ \\
$59$ & $4+4a$ & $1$ & $2$ & $ (2)$ & $ (2)$ \\
$61$ & $-2-8a$ & $0$ & $1$ & $(1)$ & $ (2)$ \\
$67$ & $-4+2a$ & $9+a$ & $2+3a$ & $ \lambda_{11 , 1}$ & $ \lambda_{11 , 1}$ \\
$89$ & $-8-4a$ & $7+9a$ & $2+3a$ & $ \lambda_{11 , 1}$ & $ \lambda_{11 , 1}$ \\
$101(**)$ & $2+4a$ & $1$ & $2$ & $ (2)$ & $ (2)$ \\
$149$ & $14+16a$ & $0$ & $1$ & $(1)$ & $ (2)$ \\
$167$ & $4-4a$ & $1$ & $2$ & $ (2)$ & $ (2)$ \\
$173$ & $18+8a$ & $1$ & $2$ & $ (2)$ & $ (2)$ \\
$199$ & $-16+6a$ & $3+3a$ & $2+3a$ & $ \lambda_{11 , 1}$ & $ \lambda_{11 , 1}$ \\
$211$ & $-16-12a$ & $1$ & $2$ & $ (2)$ & $ (2)$ \\
$223$ & $4$ & $-$ & $-$ & $-$ & $ (2)$ \\
$233$ & $-9+4a$ & $0$ & $1$ & $(1)$ & $ \lambda_{31 , 2}$ \\
\noalign{\smallskip}\hline\noalign{\smallskip}
\end{tabular}
\end{table}

\begin{table}
\begin{tabular}{p{1.5cm}p{1.7cm}p{1.5cm}p{1.5cm}p{1.5cm}p{1.5cm}}
\hline\noalign{\smallskip}
$p$ & $a_p$ & $u_p$ & $b_p$ & $\textrm{Fac}(\b_p)$ & $\textrm{Fac}(\b_{O_L})$  \\
\hline\noalign{\smallskip}
$271$ & $8$ & $-$ & $-$ & $-$ & $ (2)$ \\
$307$ & $12-4a$ & $1$ & $2$ & $ (2)$ & $ (2)$ \\
$311$ & $7+10a$ & $0$ & $1$ & $(1)$ & $ \lambda_5$ \\
$317$ & $18+12a$ & $1$ & $2$ & $ (2)$ & $ (2)$ \\
$331$ & $-11-14a$ & $4a$ & $2+3a$ & $ \lambda_{11 , 1}$ & $ \lambda_{11 , 1}$ \\
$347$ & $-16a$ & $1$ & $2$ & $ (2)$ & $ (2)$ \\
$353$ & $-3+20a$ & $4+10a$ & $2+3a$ & $ \lambda_{11 , 1}$ & $ \lambda_{11 , 1}$ \\
$379$ & $12+20a$ & $0$ & $1$ & $(1)$ & $ (2)$ \\
$383$ & $12-8a$ & $0$ & $1$ & $(1)$ & $ (2)$ \\
$397$ & $-17-12a$ & $8+5a$ & $2+3a$ & $ \lambda_{11 , 1}$ & $ \lambda_{11 , 1}$ \\
$401$ & $-8-10a$ & $0$ & $1$ & $(1)$ & $ \lambda_5$ \\
$409$ & $9+20a$ & $0$ & $1$ & $(1)$ & $ \lambda_5$ \\
$419$ & $-12+12a$ & $10+8a$ & $2+3a$ & $ \lambda_{11 , 1}$ & $ (2)\lambda_{11 , 1}$ \\
$431$ & $-20+4a$ & $0$ & $1$ & $(1)$ & $ (2)$ \\
$449$ & $-10-8a$ & $1$ & $2$ & $ (2)$ & $ (2)$ \\
$463$ & $-20$ & $-$ & $-$ & $-$ & $ (2)\lambda_{11 , 2}\lambda_{11 , 1}$ \\
$563$ & $-28-8a$ & $0$ & $1$ & $(1)$ & $ (2)$ \\
$569$ & $-16-10a$ & $0$ & $1$ & $(1)$ & $ \lambda_5$ \\
$593$ & $2-8a$ & $1$ & $2$ & $ (2)$ & $ (2)$ \\
$599$ & $24+16a$ & $1$ & $2$ & $ (2)$ & $ (2)$ \\
$607$ & $24+4a$ & $1$ & $2$ & $ (2)$ & $ (2)$ \\
$617$ & $-10+4a$ & $8+5a$ & $2+3a$ & $ \lambda_{11 , 1}$ & $ (2)\lambda_{11 , 1}$ \\
$619$ & $12+12a$ & $0$ & $1$ & $(1)$ & $ (2)$ \\
$631$ & $20a$ & $0$ & $1$ & $(1)$ & $ (2)$ \\
$661$ & $-18-8a$ & $5+6a$ & $2+3a$ & $ \lambda_{11 , 1}$ & $ (2)\lambda_{11 , 1}$ \\
$677$ & $18$ & $-$ & $-$ & $-$ & $ (2)$ \\
$683$ & $13+22a$ & $1$ & $2+3a$ & $ \lambda_{11 , 1}$ & $ \lambda_{11 , 1}$ \\
$691$ & $12-8a$ & $1$ & $2$ & $ (2)$ & $ (2)$ \\
$719$ & $-8+8a$ & $1$ & $2$ & $ (2)$ & $ (2)\lambda_{11 , 2}$ \\
$727$ & $-24-6a$ & $10+8a$ & $2+3a$ & $ \lambda_{11 , 1}$ & $ \lambda_{11 , 1}$ \\
$751$ & $-12+20a$ & $0$ & $1$ & $(1)$ & $ (2)\lambda_5$ \\
$787$ & $32-12a$ & $0$ & $1$ & $(1)$ & $ (2)$ \\
$797$ & $-22-20a$ & $0$ & $1$ & $(1)$ & $ (2)$ \\
$809(**)$ & $22-16a$ & $1$ & $2$ & $ (2)$ & $ (2)$ \\
$821$ & $-34-8a$ & $1$ & $2$ & $ (2)$ & $ (2)^2$ \\
$827$ & $4-4a$ & $0$ & $1$ & $(1)$ & $ (2)$ \\
$829$ & $18+36a$ & $1$ & $2$ & $ (2)$ & $ (2)$ \\
$853$ & $-18+12a$ & $1$ & $2$ & $ (2)$ & $ (2)$ \\
$859$ & $-13-6a$ & $10+8a$ & $2+3a$ & $ \lambda_{11 , 1}$ & $ \lambda_{11 , 1}$ \\
$877$ & $-34-4a$ & $1$ & $2$ & $ (2)$ & $ (2)$ \\
$881$ & $38+10a$ & $8+5a$ & $2+3a$ & $ \lambda_{11 , 1}$ & $ \lambda_5\lambda_{11 , 1}$ \\
$883$ & $4$ & $-$ & $-$ & $-$ & $ (2)$ \\
\noalign{\smallskip}\hline\noalign{\smallskip}
\end{tabular}
\end{table}

\begin{table}
\begin{tabular}{p{1.5cm}p{1.7cm}p{1.5cm}p{1.5cm}p{1.5cm}p{1.5cm}}
\hline\noalign{\smallskip}
$p$ & $a_p$ & $u_p$ & $b_p$ & $\textrm{Fac}(\b_p)$ & $\textrm{Fac}(\b_{O_L})$  \\
\hline\noalign{\smallskip}
$911(**)$ & $14+28a$ & $1+2a$ & $3$ & $ (3)$ & $ (3)$ \\
$941$ & $-2+14a$ & $0$ & $1$ & $(1)$ & $ (3)$ \\
$947$ & $-17+10a$ & $8+5a$ & $2+3a$ & $ \lambda_{11 , 1}$ & $ \lambda_{11 , 1}$ \\
$953$ & $18+4a$ & $0$ & $1$ & $(1)$ & $ (2)$ \\
$991$ & $24$ & $-$ & $-$ & $-$ & $ (2)\lambda_{11 , 2}\lambda_{11 , 1}$ \\
$997$ & $2+24a$ & $1$ & $2$ & $ (2)$ & $ (2)$ \\
$1009(**)$ & $6+12a$ & $0$ & $1$ & $(1)$ & $ (2)$ \\
$1013$ & $-29-8a$ & $2+7a$ & $2+3a$ & $ \lambda_{11 , 1}$ & $ \lambda_{11 , 1}$ \\
$1069$ & $26+18a$ & $0$ & $1$ & $(1)$ & $ (3)$ \\
$1091$ & $4+40a$ & $0$ & $1$ & $(1)$ & $ (2)$ \\
$1097$ & $-18-24a$ & $1$ & $2$ & $ (2)$ & $ (2)$ \\
$1117$ & $14-28a$ & $1$ & $2$ & $ (2)$ & $ (2)$ \\
$1123$ & $-34+12a$ & $5+6a$ & $2+3a$ & $ \lambda_{11 , 1}$ & $ \lambda_{11 , 1}$ \\
$1151(**)$ & $-24-48a$ & $1$ & $2$ & $ (2)$ & $ (2)$ \\
$1163$ & $-8-20a$ & $1$ & $2$ & $ (2)$ & $ (2)$ \\
$1171$ & $16-18a$ & $0$ & $1$ & $(1)$ & $ (3)$ \\
$1181$ & $-2-16a$ & $1$ & $2$ & $ (2)$ & $ (2)$ \\
$1213$ & $28+36a$ & $0$ & $1$ & $(1)$ & $ (3)$ \\
$1217$ & $4-28a$ & $0$ & $1$ & $(1)$ & $ (3)$ \\
$1231$ & $-16-24a$ & $1$ & $2$ & $ (2)$ & $ (2)$ \\
$1259$ & $-24-12a$ & $0$ & $1$ & $(1)$ & $ (2)$ \\
$1277$ & $-7-8a$ & $2+7a$ & $2+3a$ & $ \lambda_{11 , 1}$ & $ \lambda_{11 , 1}$ \\
$1279$ & $-24-42a$ & $10+a$ & $1+3a$ & $ \lambda_{11 , 2}$ & $ \lambda_{11 , 2}$ \\
$1301$ & $47+4a$ & $0$ & $1$ & $(1)$ & $ (3)$ \\
$1303$ & $12+20a$ & $0$ & $1$ & $(1)$ & $ (2)$ \\
$1319$ & $4-16a$ & $1$ & $2$ & $ (2)$ & $ (2)$ \\
$1321$ & $8-24a$ & $4+10a$ & $2+3a$ & $ \lambda_{11 , 1}$ & $ \lambda_{11 , 1}$ \\
$1409$ & $-31-44a$ & $1$ & $2+3a$ & $ \lambda_{11 , 1}$ & $ \lambda_{11 , 1}$ \\
$1451$ & $-8+32a$ & $1$ & $2$ & $ (2)$ & $ (2)$ \\
$1453$ & $2$ & $-$ & $-$ & $-$ & $ (2)\lambda_{11 , 2}\lambda_{11 , 1}$ \\
$1459$ & $66+10a$ & $0$ & $1$ & $(1)$ & $ \lambda_5$ \\
$1481$ & $6+8a$ & $1$ & $2$ & $ (2)$ & $ (2)$ \\
$1483$ & $-36+8a$ & $0$ & $1$ & $(1)$ & $ (2)$ \\
$1489$ & $-4+36a$ & $0$ & $1$ & $(1)$ & $ (3)$ \\
$1499$ & $-13+2a$ & $0$ & $1$ & $(1)$ & $ \lambda_{11 , 2}$ \\
$1523$ & $-24-56a$ & $0$ & $1$ & $(1)$ & $ (2)$ \\
$1543$ & $-41-18a$ & $2$ & $3$ & $ (3)$ & $ (3)$ \\
$1549$ & $43$ & $-$ & $-$ & $-$ & $ (3)$ \\
$1553$ & $-6-8a$ & $1$ & $2$ & $ (2)$ & $ (2)$ \\
$1559$ & $39-10a$ & $0$ & $1$ & $(1)$ & $ \lambda_5$ \\
$1607$ & $-46-28a$ & $10+8a$ & $2+3a$ & $ \lambda_{11 , 1}$ & $ \lambda_{11 , 1}$ \\
$1613$ & $18+48a$ & $1$ & $2$ & $ (2)$ & $ (2)$ \\
\noalign{\smallskip}\hline\noalign{\smallskip}
\end{tabular}
\end{table}

\begin{table}[ht]
\begin{tabular}{p{1.5cm}p{1.7cm}p{1.5cm}p{1.5cm}p{1.5cm}p{1.5cm}}
\hline\noalign{\smallskip}
$p$ & $a_p$ & $u_p$ & $b_p$ & $\textrm{Fac}(\b_p)$ & $\textrm{Fac}(\b_{O_L})$  \\
\hline\noalign{\smallskip}
$1663$ & $-8-44a$ & $0$ & $1$ & $(1)$ & $ (2)$ \\
$1667$ & $-36-48a$ & $0$ & $1$ & $(1)$ & $ (2)$ \\
$1669$ & $-38-32a$ & $3$ & $4$ & $ (2)^2$ & $ (2)$ \\
$1697$ & $-38-8a$ & $1$ & $2$ & $ (2)$ & $ (2)$ \\
$1721(**)$ & $4+8a$ & $0$ & $1$ & $(1)$ & $ (3)^2$ \\
$1733$ & $-47-4a$ & $0$ & $1$ & $(1)$ & $ (3)$ \\
$1783$ & $-57-6a$ & $10+8a$ & $2+3a$ & $ \lambda_{11 , 1}$ & $ \lambda_{11 , 1}$ \\
$1787$ & $40-4a$ & $1$ & $2$ & $ (2)$ & $ (2)$ \\
$1789$ & $-18+16a$ & $1$ & $2$ & $ (2)$ & $ (2)$ \\
$1811$ & $28+52a$ & $0$ & $1$ & $(1)$ & $ (2)$ \\
$1831$ & $-52-10a$ & $0$ & $1$ & $(1)$ & $ \lambda_5$ \\
$1861$ & $-30-44a$ & $0$ & $1$ & $(1)$ & $ (2)$ \\
$1867$ & $20+44a$ & $1$ & $2$ & $ (2)$ & $ (2)$ \\
$1871$ & $-12+12a$ & $21+8a$ & $4+6a$ & $ (2)\lambda_{11 , 1}$ & $ (2)\lambda_{11 , 1}$ \\
$1873$ & $-38-8a$ & $0$ & $1$ & $(1)$ & $ (2)$ \\
$1877$ & $22+32a$ & $0$ & $1$ & $(1)$ & $ (2)^2$ \\
$1879$ & $-20-12a$ & $1$ & $2$ & $ (2)$ & $ (2)$ \\
$1889$ & $-2-44a$ & $1$ & $2$ & $ (2)$ & $ (2)$ \\
$1901$ & $-14+8a$ & $2+a$ & $3$ & $ (3)$ & $ (2)(3)$ \\
$1913$ & $-62-16a$ & $1$ & $2$ & $ (2)$ & $ (2)$ \\
$1931$ & $2+20a$ & $0$ & $1$ & $(1)$ & $ \lambda_5$ \\
$1949$ & $-58-20a$ & $0$ & $1$ & $(1)$ & $ (2)$ \\
$1997$ & $-46-8a$ & $0$ & $1$ & $(1)$ & $ (2)$ \\
\noalign{\smallskip}\hline\noalign{\smallskip}
\end{tabular}
%$^a$ Table foot note (with superscript)
\end{table}

\subsection{Second example}

Let now $f\in S_2(\Gamma_0(125))$ be the normalized cusp form of weight $2$ and level $125$ lying in the first Galois orbit of eigenforms. The element $a\in E_f$ has also
in this case minimal polynomial $x^2+x-1$. The first few coefficients of the Fourier expansion of $f$ are
\[
f= q + aq^2 - (a+2)q^3 - (a+1)q^4 -(a+1)q^6+\ldots
\]
Consider the Galois representation
\begin{equation}\label{eq2tor}
\bar\rho_{(2)}:\gq\longrightarrow\aut_{O_{E_f}/(2)}(A_f[2])\simeq\gl_2(O_{E_f}/(2)).
\end{equation}
defined by the $2$-torsion $A_f[2]$ of $A_f$. The rational prime $2$ is inert in $E_f\simeq\qu(\sqrt 5)$, denote by $\f_4$ its residue field.
Since $\bar\rho_{(2)}$ has trivial determinant we see that $\bar\rho_{(2)}$ is valued in the special linear group $\sl_2(\f)$, which is isomorphic to
$A_5$, the alternating group in five letters.

After analyzing the reduction modulo $2$ of the first few Hecke eigenvalues of $f$, and using elementary group theory, one can deduce that
\begin{equation}\label{eq:surjA5}
\textrm{Im}(\bar\rho_{(2)})\simeq \sl_2(\f),
\end{equation}
i.e., $\bar\rho_{(2)}$ defines an $A_5$-extension $K/\qu$. According to Corollary~\label{cor:split}, we have that a rational prime $p\nmid 2\cdot 5$
splits completely in $K$ if and only if $(2)$ divides $\b_p$, which, by Chebotarev, happens for a set of primes of density $1/60\sim 0,017$.
In Table~\ref{tab:2} we observe this splitting phenomenon for $p=887, 1657$ and $1699$.

Lastly, notice that for every prime $p\equiv 1\text{ mod $5$}$ for which we were able to compute $\sigma_p$, we have that the unique prime
of $E_f$ lying above $5$ divides $\b_p$. Reasoning as in the first example, this suggest that there is a decomposition
\[
\bar\rho_{\lambda_5}\simeq \textrm{1}\oplus\chi_5,
\]
where $\chi_5$ denotes the mod $5$ cyclotomic character. However, with our methods we are not able to prove this.

\begin{table}[ht]
\caption{Integral Frobenius for $A_f$, where $f$ lies in the first Galois orbit of $S_2(\Gamma_0(125))$}\label{tab:2}
\begin{tabular}{p{1.5cm}p{1.7cm}p{1.5cm}p{1.5cm}p{1.5cm}p{1.5cm}}
\hline\noalign{\smallskip}
$p$ & $a_p$ & $u_p$ & $b_p$ & $\textrm{Fac}(\b_p)$ & $\textrm{Fac}(\b_{O_L})$  \\
\hline\noalign{\smallskip}
$11$ & $-3$ & $-$ & $-$ & $-$ & $ \lambda_5$ \\
$31(*)$ & $-3-5a$ & $1$ & $1+2a$ & $ \lambda_5$ & $ \lambda_5$ \\
$41$ & $-3$ & $-$ & $-$ & $-$ & $ \lambda_5$ \\
$61$ & $2+5a$ & $-$ & $-$ & $-$ & $ \lambda_5\lambda_{19 , 1}$ \\
$71$ & $-3$ & $-$ & $-$ & $-$ & $ \lambda_5^2$ \\
$89(**)$ & $6+12a$ & $0$ & $1$ & $(1)$ & $ (2)$ \\
$101$ & $-3$ & $-$ & $-$ & $-$ & $ \lambda_5$ \\
$131$ & $12+15a$ & $1$ & $1+2a$ & $ \lambda_5$ & $ \lambda_5$ \\
$137$ & $4-a$ & $2+a$ & $3$ & $ (3)$ & $ (3)$ \\
$151$ & $-13+5a$ & $1$ & $1+2a$ & $ \lambda_5$ & $ \lambda_5$ \\
$173$ & $-13-8a$ & $1+2a$ & $3$ & $ (3)$ & $ (3)$ \\
$181$ & $2-10a$ & $1$ & $1+2a$ & $ \lambda_5$ & $ \lambda_5$ \\
$191$ & $12$ & $-$ & $-$ & $-$ & $ (2)\lambda_5$ \\
$211$ & $12+10a$ & $1$ & $1+2a$ & $ \lambda_5$ & $ \lambda_5$ \\
$229$ & $-3-a$ & $0$ & $1$ & $(1)$ & $ \lambda_{11 , 1}$ \\
$233$ & $-1+16a$ & $1+2a$ & $3$ & $ (3)$ & $ (3)$ \\
$241$ & $-3+10a$ & $1$ & $1+2a$ & $ \lambda_5$ & $ \lambda_5$ \\
$251$ & $-18-15a$ & $1$ & $1+2a$ & $ \lambda_5$ & $ \lambda_5$ \\
$271$ & $12-5a$ & $1$ & $1+2a$ & $ \lambda_5$ & $ \lambda_5$ \\
$281$ & $12$ & $-$ & $-$ & $-$ & $ \lambda_5(7)$ \\
$311$ & $-3-15a$ & $1$ & $1+2a$ & $ \lambda_5$ & $ \lambda_5$ \\
$313$ & $-12a$ & $0$ & $1$ & $(1)$ & $ \lambda_{11 , 1}$ \\
$317$ & $-14+8a$ & $2+a$ & $3$ & $ (3)$ & $ (2)(3)$ \\
$331$ & $-13+5a$ & $1$ & $1+2a$ & $ \lambda_5$ & $ \lambda_5$ \\
$353$ & $-22-16a$ & $0$ & $1$ & $(1)$ & $ (2)$ \\
$379$ & $7+9a$ & $2$ & $3$ & $ (3)$ & $ (3)$ \\
$401$ & $12$ & $-$ & $-$ & $-$ & $ \lambda_5$ \\
$421$ & $17+5a$ & $1$ & $1+2a$ & $ \lambda_5$ & $ \lambda_5$ \\
$431$ & $12+15a$ & $1$ & $1+2a$ & $ \lambda_5$ & $ \lambda_5$ \\
$439$ & $1-18a$ & $2$ & $3$ & $ (3)$ & $ (3)$ \\
\noalign{\smallskip}\hline\noalign{\smallskip}
\end{tabular}
%$^a$ Table foot note (with superscript)
\end{table}

\begin{table}
\begin{tabular}{p{1.5cm}p{1.7cm}p{1.5cm}p{1.5cm}p{1.5cm}p{1.5cm}}
\hline\noalign{\smallskip}
$p$ & $a_p$ & $u_p$ & $b_p$ & $\textrm{Fac}(b_p)$ & $\textrm{Fac}(b_{O_L})$  \\
\hline\noalign{\smallskip}
$457$ & $-18$ & $-$ & $-$ & $-$ & $ (2)^2$ \\
$461$ & $12-15a$ & $1$ & $1+2a$ & $ \lambda_5$ & $ \lambda_5$ \\
$491$ & $12-15a$ & $1$ & $1+2a$ & $ \lambda_5$ & $ \lambda_5$ \\
$503$ & $8-11a$ & $1+2a$ & $3$ & $ (3)$ & $ (3)$ \\
$509(**)$ & $-6-12a$ & $0$ & $1$ & $(1)$ & $ (2)$ \\
$521$ & $-18-15a$ & $1$ & $1+2a$ & $ \lambda_5$ & $ \lambda_5$ \\
$541$ & $-18+10a$ & $1$ & $1+2a$ & $ \lambda_5$ & $ \lambda_5$ \\
$547$ & $-27-3a$ & $-$ & $-$ & $-$ & $ \lambda_{59 , 1}$ \\
$557$ & $-8+20a$ & $2+a$ & $3$ & $ (3)$ & $ (3)$ \\
$563$ & $20+8a$ & $0$ & $1$ & $(1)$ & $ (2)$ \\
$571$ & $-13$ & $-$ & $-$ & $-$ & $ \lambda_5(3)$ \\
$587$ & $4+4a$ & $0$ & $1$ & $(1)$ & $ (2)$ \\
$601$ & $-33-20a$ & $1$ & $1+2a$ & $ \lambda_5$ & $ \lambda_5$ \\
$631$ & $2$ & $-$ & $-$ & $-$ & $ \lambda_5(3)$ \\
$641$ & $-33-15a$ & $1$ & $1+2a$ & $ \lambda_5$ & $ \lambda_5$ \\
$647$ & $-17+2a$ & $2+a$ & $3$ & $ (3)$ & $ (3)$ \\
$661$ & $-18-20a$ & $2+2a$ & $1+2a$ & $ \lambda_5$ & $ (2)\lambda_5$ \\
$677$ & $30+16a$ & $0$ & $1$ & $(1)$ & $ (2)$ \\
$691$ & $42+10a$ & $1$ & $1+2a$ & $ \lambda_5$ & $ \lambda_5$ \\
$701$ & $27+15a$ & $1$ & $1+2a$ & $ \lambda_5$ & $ \lambda_5$ \\
$727$ & $-24a$ & $0$ & $1$ & $(1)$ & $ (2)$ \\
$743$ & $-34-5a$ & $1+2a$ & $3$ & $ (3)$ & $ (3)$ \\
$751$ & $17+20a$ & $1$ & $1+2a$ & $ \lambda_5$ & $ \lambda_5$ \\
$757$ & $27$ & $-$ & $-$ & $-$ & $ \lambda_{11 , 2}\lambda_{11 , 1}$ \\
$761$ & $-18$ & $-$ & $-$ & $-$ & $ (2)^3\lambda_5$ \\
$811$ & $-28-10a$ & $1$ & $1+2a$ & $ \lambda_5$ & $ \lambda_5$ \\
$821$ & $-3-15a$ & $1$ & $1+2a$ & $ \lambda_5$ & $ \lambda_5$ \\
$859$ & $4+18a$ & $2$ & $3$ & $ (3)$ & $ (3)$ \\
$863$ & $-10-2a$ & $1+2a$ & $3$ & $ (3)$ & $ (3)$ \\
$881$ & $-3-30a$ & $1$ & $1+2a$ & $ \lambda_5$ & $ \lambda_5$ \\
$887$ & $-36+4a$ & $1$ & $2$ & $ (2)$ & $ (2)$ \\
$911$ & $12+30a$ & $1$ & $1+2a$ & $ \lambda_5$ & $ \lambda_5$ \\
$941$ & $-3+15a$ & $1$ & $1+2a$ & $ \lambda_5$ & $ \lambda_5$ \\
$971$ & $-3+30a$ & $1$ & $1+2a$ & $ \lambda_5$ & $ \lambda_5$ \\
$991$ & $-43-10a$ & $1$ & $1+2a$ & $ \lambda_5$ & $ \lambda_5$ \\
$1021$ & $17+20a$ & $1$ & $1+2a$ & $ \lambda_5$ & $ \lambda_5$ \\
$1031$ & $-3-30a$ & $1$ & $1+2a$ & $ \lambda_5$ & $ \lambda_5$ \\
$1051$ & $-28-45a$ & $1$ & $3+6a$ & $ \lambda_5(3)$ & $ \lambda_5(3)$ \\
$1061$ & $27+30a$ & $1$ & $1+2a$ & $ \lambda_5$ & $ \lambda_5$ \\
$1091$ & $-3-30a$ & $1$ & $1+2a$ & $ \lambda_5$ & $ \lambda_5$ \\
$1097$ & $-17+2a$ & $2+a$ & $3$ & $ (3)$ & $ (3)$ \\
$1151$ & $12+45a$ & $1$ & $1+2a$ & $ \lambda_5$ & $ \lambda_5$ \\
\noalign{\smallskip}\hline\noalign{\smallskip}
\end{tabular}
%$^a$ Table foot note (with superscript)
\end{table}

\begin{table}
\begin{tabular}{p{1.5cm}p{1.7cm}p{1.5cm}p{1.5cm}p{1.5cm}p{1.5cm}}
\hline\noalign{\smallskip}
$p$ & $a_p$ & $u_p$ & $b_p$ & $\textrm{Fac}(b_p)$ & $\textrm{Fac}(b_{O_L})$  \\
\hline\noalign{\smallskip}
$1171$ & $-3+25a$ & $1$ & $1+2a$ & $ \lambda_5$ & $ \lambda_5^2$ \\
$1181$ & $-18-15a$ & $1$ & $1+2a$ & $ \lambda_5$ & $ \lambda_5$ \\
$1193$ & $-24+20a$ & $0$ & $1$ & $(1)$ & $ \lambda_{11 , 1}$ \\
$1201$ & $-3-5a$ & $1$ & $1+2a$ & $ \lambda_5$ & $ \lambda_5$ \\
$1231$ & $-18-5a$ & $1$ & $1+2a$ & $ \lambda_5$ & $ \lambda_5$ \\
$1291$ & $2+20a$ & $1$ & $1+2a$ & $ \lambda_5$ & $ \lambda_5$ \\
$1301$ & $-18+30a$ & $1$ & $1+2a$ & $ \lambda_5$ & $ \lambda_5$ \\
$1321$ & $2-40a$ & $2+2a$ & $1+2a$ & $ \lambda_5$ & $ (2)\lambda_5$ \\
$1361$ & $42$ & $-$ & $-$ & $-$ & $ (2)\lambda_5$ \\
$1367$ & $1-7a$ & $2+a$ & $3$ & $ (3)$ & $ (3)$ \\
$1381$ & $27+25a$ & $1$ & $1+2a$ & $ \lambda_5$ & $ \lambda_5$ \\
$1399$ & $-50$ & $-$ & $-$ & $-$ & $ (3)$ \\
$1433$ & $14-4a$ & $0$ & $1$ & $(1)$ & $ (2)$ \\
$1451$ & $12$ & $-$ & $-$ & $-$ & $ (2)\lambda_5$ \\
$1471$ & $-18-20a$ & $1$ & $1+2a$ & $ \lambda_5$ & $ \lambda_5$ \\
$1481$ & $-48-15a$ & $1$ & $1+2a$ & $ \lambda_5$ & $ \lambda_5$ \\
$1511$ & $-3+30a$ & $1$ & $1+2a$ & $ \lambda_5$ & $ \lambda_5$ \\
$1531$ & $2$ & $-$ & $-$ & $-$ & $ \lambda_5(3)$ \\
$1549$ & $-20-45a$ & $2$ & $3$ & $ (3)$ & $ (3)$ \\
$1571$ & $-18+30a$ & $1$ & $1+2a$ & $ \lambda_5$ & $ \lambda_5$ \\
$1583$ & $-1+16a$ & $1+2a$ & $3$ & $ (3)$ & $ (3)$ \\
$1601$ & $-3+15a$ & $1$ & $1+2a$ & $ \lambda_5$ & $ \lambda_5$ \\
$1607$ & $-48-20a$ & $0$ & $1$ & $(1)$ & $ (2)$ \\
$1621$ & $47+45a$ & $1$ & $3+6a$ & $ \lambda_5(3)$ & $ \lambda_5(3)$ \\
$1657$ & $42+60a$ & $1$ & $2$ & $ (2)$ & $ (2)$ \\
$1663$ & $-60-12a$ & $0$ & $1$ & $(1)$ & $ (2)$ \\
$1669$ & $-32-9a$ & $2$ & $3$ & $ (3)$ & $ (3)$ \\
$1699$ & $40+40a$ & $1$ & $2$ & $ (2)$ & $ (2)$ \\
$1721$ & $-3-30a$ & $1$ & $1+2a$ & $ \lambda_5$ & $ \lambda_5$ \\
$1741$ & $-13+20a$ & $1$ & $1+2a$ & $ \lambda_5$ & $ \lambda_5$ \\
$1777$ & $-30-24a$ & $0$ & $1$ & $(1)$ & $ (2)$ \\
$1789$ & $1-18a$ & $2$ & $3$ & $ (3)$ & $ (3)$ \\
$1801$ & $2-10a$ & $1$ & $1+2a$ & $ \lambda_5$ & $ \lambda_5$ \\
$1811$ & $-63$ & $-$ & $-$ & $-$ & $ \lambda_5^2$ \\
$1823$ & $-43-23a$ & $-$ & $-$ & $-$ & $ (3)$ \\
$1831$ & $12-35a$ & $1$ & $1+2a$ & $ \lambda_5$ & $ \lambda_5$ \\
$1861$ & $17+45a$ & $1$ & $3+6a$ & $ \lambda_5(3)$ & $ \lambda_5(3)$ \\
$1871$ & $27$ & $-$ & $-$ & $-$ & $ \lambda_5$ \\
$1901$ & $27+30a$ & $1$ & $1+2a$ & $ \lambda_5$ & $ \lambda_5$ \\
$1931$ & $27+45a$ & $1$ & $1+2a$ & $ \lambda_5$ & $ \lambda_5$ \\
$1951$ & $-33-20a$ & $1$ & $1+2a$ & $ \lambda_5$ & $ \lambda_5$ \\
\noalign{\smallskip}\hline\noalign{\smallskip}
\end{tabular}
%$^a$ Table foot note (with superscript)
\end{table}

\subsection{Third example}

In our last example we consider a normalized cuspidal $f\in S_2(\Gamma_0(133))$ of weight $2$ and conductor $133$ lying in the first Galois orbit of eigenforms.
The first few coefficients of the Fourier expansion of $f$ are
\[
f= q + aq^2 +aq^3 -3 (a+1)q^4 - (2a+3)q^5-(3a+1)q^6+\ldots
\]
where $a\in E_f$ has minimal polynomial $x^2+3x+1$. The same argument used in the second example shows that $\bar\rho_{(2)}$ defines an $A_5$-extension
$K/\qu$. Looking at Table~\ref{tab:3}, we observe that the primes $839, 941, 1663, 1783$ and $1789$ are completely split in $K$.

\begin{table}[ht]
\caption{Integral Frobenius for $A_f$, where $f$ lies in the first Galois orbit of $S_2(\Gamma_0(133))$}\label{tab:3}
\begin{tabular}{p{1.5cm}p{1.7cm}p{1.5cm}p{1.5cm}p{1.5cm}p{1.5cm}}
\hline\noalign{\smallskip}
$p$ & $a_p$ & $u_p$ & $b_p$ & $\textrm{Fac}(b_p)$ & $\textrm{Fac}(b_{O_L})$  \\
\hline\noalign{\smallskip}
$5$ & $-3-2a$ & $-$ & $-$ & $-$ & $ \lambda_5$ \\
$11$ & $-3+a$ & $0$ & $1$ & $(1)$ & $ (3)$ \\
$29$ & $-3+a$ & $0$ & $1$ & $(1)$ & $ (3)$ \\
$47(**)$ & $-15-10a$ & $a$ & $3$ & $ (3)$ & $ (3)$ \\
$59$ & $-15-6a$ & $0$ & $1$ & $(1)$ & $ \lambda_{11 , 1}$ \\
$79$ & $-10$ & $-$ & $-$ & $-$ & $ (3)$ \\
$131$ & $-3-5a$ & $0$ & $1$ & $(1)$ & $ \lambda_5(3)$ \\
$137$ & $6-4a$ & $0$ & $1$ & $(1)$ & $ (2)$ \\
$173$ & $6+10a$ & $0$ & $1$ & $(1)$ & $ (3)$ \\
$181$ & $-25-9a$ & $0$ & $1$ & $(1)$ & $ (3)$ \\
$193$ & $-5-9a$ & $2$ & $3$ & $ (3)$ & $ (3)$ \\
$229$ & $-14-12a$ & $0$ & $1$ & $(1)$ & $ (2)$ \\
$239$ & $15+4a$ & $2a$ & $3$ & $ (3)$ & $ (3)$ \\
$251$ & $-3-11a$ & $0$ & $1$ & $(1)$ & $ (3)$ \\
$311$ & $3-5a$ & $4$ & $3+2a$ & $ \lambda_5$ & $ \lambda_5$ \\
$317$ & $30+12a$ & $0$ & $1$ & $(1)$ & $ (2)$ \\
$389$ & $15+11a$ & $0$ & $1$ & $(1)$ & $ \lambda_{19 , 2}$ \\
$431(**)$ & $-30-20a$ & $0$ & $1$ & $(1)$ & $ (3)$ \\
$439$ & $8$ & $-$ & $-$ & $-$ & $ (2)(3)$ \\
$443$ & $-12+a$ & $2a$ & $3$ & $ (3)$ & $ (3)$ \\
$449$ & $-12+7a$ & $2a$ & $3$ & $ (3)$ & $ (3)$ \\
$457$ & $28+9a$ & $0$ & $1$ & $(1)$ & $ (3)$ \\
$479$ & $51+25a$ & $0$ & $1$ & $(1)$ & $ \lambda_5$ \\
$491$ & $-12-4a$ & $0$ & $1$ & $(1)$ & $ (2)$ \\
$503(**)$ & $24+16a$ & $0$ & $1$ & $(1)$ & $ (2)(3)$ \\
$509$ & $30$ & $-$ & $-$ & $-$ & $ (2)$ \\
$541(**)$ & $18+12a$ & $0$ & $1$ & $(1)$ & $ (2)$ \\
$571$ & $-23-18a$ & $0$ & $1$ & $(1)$ & $ (3)$ \\
$599$ & $-6-5a$ & $2$ & $3+2a$ & $ \lambda_5$ & $ \lambda_5\lambda_{11 , 1}$ \\
$619$ & $10-9a$ & $0$ & $1$ & $(1)$ & $ (3)$ \\
\noalign{\smallskip}\hline\noalign{\smallskip}
\end{tabular}
%$^a$ Table foot note (with superscript)
\end{table}

\begin{table}
\begin{tabular}{p{1.5cm}p{1.7cm}p{1.5cm}p{1.5cm}p{1.5cm}p{1.5cm}}
\hline\noalign{\smallskip}
$p$ & $a_p$ & $u_p$ & $b_p$ & $\textrm{Fac}(b_p)$ & $\textrm{Fac}(b_{O_L})$  \\
\hline\noalign{\smallskip}
$631$ & $1+6a$ & $0$ & $1$ & $(1)$ & $ \lambda_{11 , 1}$ \\
$661$ & $-26-24a$ & $0$ & $1$ & $(1)$ & $ (2)$ \\
$677$ & $-42-19a$ & $0$ & $1$ & $(1)$ & $ (3)$ \\
$719$ & $12+16a$ & $0$ & $1$ & $(1)$ & $ (2)$ \\
$757$ & $10+12a$ & $0$ & $1$ & $(1)$ & $ (2)$ \\
$787$ & $20+12a$ & $0$ & $1$ & $(1)$ & $ (2)$ \\
$839$ & $24+20a$ & $1$ & $2$ & $ (2)$ & $ (2)\lambda_5$ \\
$857$ & $-69-37a$ & $a$ & $3$ & $ (3)$ & $ (3)$ \\
$911$ & $-6+8a$ & $a$ & $3$ & $ (3)$ & $ (3)$ \\
$941$ & $6-20a$ & $1$ & $2$ & $ (2)$ & $ (2)$ \\
$971$ & $-33-10a$ & $0$ & $1$ & $(1)$ & $ \lambda_5$ \\
$977$ & $10a$ & $-$ & $-$ & $-$ & $ \lambda_{89 , 1}$ \\
$983$ & $-57-26a$ & $2a$ & $3$ & $ (3)$ & $ (3)$ \\
$1051$ & $38+15a$ & $0$ & $1$ & $(1)$ & $ \lambda_5$ \\
$1061(**)$ & $3+2a$ & $a$ & $3$ & $ (3)$ & $ (3)$ \\
$1087$ & $37+27a$ & $2$ & $3$ & $ (3)$ & $ (3)$ \\
$1109$ & $6-4a$ & $0$ & $1$ & $(1)$ & $ (2)$ \\
$1117$ & $-7+18a$ & $0$ & $1$ & $(1)$ & $ (3)$ \\
$1217$ & $33-8a$ & $2a$ & $3$ & $ (3)$ & $ (3)$ \\
$1231$ & $-46-27a$ & $0$ & $1$ & $(1)$ & $ (3)$ \\
$1249$ & $-44-18a$ & $2$ & $3$ & $ (3)$ & $ (3)$ \\
$1259$ & $21+25a$ & $3$ & $3+2a$ & $ \lambda_5$ & $ \lambda_5$ \\
$1303$ & $-54-30a$ & $0$ & $1$ & $(1)$ & $ \lambda_{11 , 1}$ \\
$1361$ & $6-20a$ & $0$ & $1$ & $(1)$ & $ (2)$ \\
$1367$ & $-36-8a$ & $0$ & $1$ & $(1)$ & $ (2)$ \\
$1409$ & $66+32a$ & $0$ & $1$ & $(1)$ & $ (2)(3)$ \\
$1447$ & $-56-48a$ & $0$ & $1$ & $(1)$ & $ (2)$ \\
$1451$ & $-33-37a$ & $-$ & $-$ & $-$ & $ (3)$ \\
$1483$ & $62+45a$ & $1$ & $3$ & $ (3)$ & $ (3)$ \\
$1487$ & $-84-53a$ & $-$ & $-$ & $-$ & $ (3)$ \\
$1493$ & $-54-28a$ & $0$ & $1$ & $(1)$ & $ (2)$ \\
$1531$ & $43$ & $-$ & $-$ & $-$ & $ \lambda_5^2(3)$ \\
$1553$ & $-75-a$ & $0$ & $1$ & $(1)$ & $ (7)$ \\
$1567$ & $-38-18a$ & $0$ & $1$ & $(1)$ & $ (3)\lambda_{11 , 2}$ \\
$1609$ & $-17+15a$ & $-$ & $-$ & $-$ & $ \lambda_{109 , 1}$ \\
$1663$ & $44+24a$ & $1$ & $2$ & $ (2)$ & $ (2)$ \\
$1669$ & $49+18a$ & $2$ & $3$ & $ (3)$ & $ (3)$ \\
$1723$ & $56+24a$ & $0$ & $1$ & $(1)$ & $ (2)$ \\
$1733$ & $-24-13a$ & $0$ & $1$ & $(1)$ & $ (3)$ \\
$1741$ & $7+15a$ & $0$ & $1$ & $(1)$ & $ \lambda_5$ \\
$1753$ & $-10+27a$ & $1$ & $3$ & $ (3)$ & $ (3)$ \\
$1759$ & $14-9a$ & $0$ & $1$ & $(1)$ & $ (3)$ \\
\noalign{\smallskip}\hline\noalign{\smallskip}
\end{tabular}
%$^a$ Table foot note (with superscript)
\end{table}

\begin{table}[ht]
\begin{tabular}{p{1.5cm}p{1.7cm}p{1.5cm}p{1.5cm}p{1.5cm}p{1.5cm}}
\hline\noalign{\smallskip}
$p$ & $a_p$ & $u_p$ & $b_p$ & $\textrm{Fac}(b_p)$ & $\textrm{Fac}(b_{O_L})$  \\
\hline\noalign{\smallskip}

$1783$ & $-32+12a$ & $1$ & $2$ & $ (2)$ & $ (2)$ \\
$1789$ & $-6-12a$ & $1$ & $2$ & $ (2)$ & $ (2)$ \\
$1823$ & $-84-23a$ & $-$ & $-$ & $-$ & $ (3)$ \\
$1847(**)$ & $-48-32a$ & $1$ & $2$ & $ (2)$ & $ (2)$ \\
$1871$ & $24$ & $-$ & $-$ & $-$ & $ (2)$ \\
$1873$ & $-43-18a$ & $0$ & $1$ & $(1)$ & $ (3)$ \\
$1879$ & $-35$ & $-$ & $-$ & $-$ & $ (3)$ \\
$1889$ & $-105-52a$ & $-$ & $-$ & $-$ & $ (3)$ \\
$1907$ & $-60-7a$ & $0$ & $1$ & $(1)$ & $ (3)$ \\
$1933$ & $10-9a$ & $0$ & $1$ & $(1)$ & $ (3)$ \\
$1973$ & $-6-16a$ & $0$ & $1$ & $(1)$ & $ (2)$ \\
$1987$ & $-10-27a$ & $1$ & $3$ & $ (3)$ & $ (3)$ \\
\noalign{\smallskip}\hline\noalign{\smallskip}
\end{tabular}
%$^a$ Table foot note (with superscript)
\end{table}

%\begin{acknowledgement}
%We wish to thank Gebhard B\"ockle for all the useful discussions we had on this project. We thank Gaetan Bisson for an intense email correspondence
%where he helped us understanding better his work. The first author thanks Jordi Gu\`ardia and Josep Gonz\`alez for the warm hospitality he received
%during his visit to Universitat Polit\`ecnica de Catalunya in July 2015. This work was supported by the DFG Priority Program SPP 1489 and
%the Luxembourg FNR.
%\end{acknowledgement}

\bibliographystyle{amsplain}

\end{document}